\documentclass[12pt,a4paper]{article}
\usepackage{amsmath,amsthm,amssymb, mathrsfs, tocloft}
\usepackage{hyperref}
\usepackage{graphicx}
\usepackage{cite}
\usepackage{pgfplots}
\usepackage{subcaption}
\usepackage{setspace}
\usepackage{epsfig}
\usepackage{sidecap}
\usepackage{verbatim}

\input xy
\xyoption{all}

\newtheorem{thm}{Theorem}[section]
\newtheorem{cor}[thm]{Corollary}
\newtheorem{lem}[thm]{Lemma}

\newtheorem{prop}[thm]{Proposition}

 \theoremstyle{definition}

\theoremstyle{remark}
\newtheorem{remark}[thm]{Remark}

\newtheorem{example}[thm]{Example}

\numberwithin{equation}{section}

\newcommand{\ben}{\begin{enumerate}}

\newcommand{\een}{\end{enumerate}}
\newcommand{\bit}{\begin{itemize}}
\newcommand{\eit}{\end{itemize}}

\begin{document}

\title
{Constructing Infinite Sets of Orthogonal Exponentials for Convex Polytopes}

\vskip 1cm
\author{Yehonatan Salman \\ Email: salman.yehonatan@gmail.com\\}
\date{}

\maketitle

\begin{abstract}

The aim of this article is to show the existence, and also give an explicit construction, of infinite sets of orthogonal exponentials for certain families of convex polytopes which include simple-rational polytopes and also non simple polytopes which satisfy other nontrivial conditions. We also show that by considering weight functions one can construct infinite sets of orthogonal exponentials with a positive density by considering orthogonal projections of affine transformations of hypercubes (i.e., zonotopes).

\end{abstract}

\section{Introduction and Mathematical Background}

Let $d\geq1$ be a positive integer and $\mathcal{D}$ be a bounded measurable set in $\Bbb R^{d}$ with a positive Lebesgue measure. For a countable subset $\Lambda\subset\Bbb R^{d}$, the family $E(\Lambda) = \{e^{i\langle \lambda, \cdot\rangle}\}_{\lambda\in\Lambda}$, where $\langle, \rangle$ denotes the scalar product in $\Bbb R^{d}$, is said to be \textbf{a set of orthogonal exponentials} for $\mathcal{D}$ if the following condition
\vskip-0.2cm
\begin{equation}\label{eq:1}\hskip-3.5cm\int_{\mathcal{D}}e^{i\langle \lambda, x\rangle}\overline{e^{i\langle \lambda', x\rangle}}dx = \int_{\mathcal{D}}e^{i\langle\lambda - \lambda', x\rangle}dx = 0, \forall\lambda, \lambda'\in\Lambda, \lambda\neq\lambda'\end{equation}
is satisfied.

The problem of existence and construction of sets $E(\Lambda)$ of orthogonal exponentials for various prescribed sets $\mathcal{D}$ was investigated by many authors where most research and results are dedicated for characterizing the family of sets $\mathcal{D}$ which have orthogonal sets of exponentials $E(\Lambda)$ and for which $\Lambda$ is also a spectrum set for $\mathcal{D}$. That is, $E(\Lambda)$ is a set of orthogonal exponentials for $\mathcal{D}$ which is also a basis for $L^{2}(\mathcal{D})$. For a small sample of some classical results see \cite{Fuglede74, Fuglede01, IKT01, IKT03, Kol, Laba, LRW}.

The results in the references mentioned above imply some nontrivial necessary conditions on the family of sets $\mathcal{D}$ which have spectrum sets $\Lambda$. For example, in \cite{Kol} it is proved that if $\mathcal{D}$ is convex and has a spectrum set then it is centrally symmetric and in \cite{IKT01} it is shown that any convex body $\mathcal{D}$ with a smooth boundary does not have spectrum sets. In \cite{Fuglede74} Fuglede made the conjecture that $\mathcal{D}$ has a spectrum set if and only if it tiles the space $\Bbb R^{d}$ by translations. Considering Fuglede's conjecture and the results mentioned above it follows that only certain families of bounded measurable sets $\mathcal{D}$ have spectrum sets and (aside from the trivial cases where $\mathcal{D}$ is an affine transformation of a hypercube) the construction of such sets will be quite complicated. Hence, the more general question arises on the existence and construction of nontrivial orthogonal sets $E(\Lambda)$ of exponentials for a given set $\mathcal{D}$ in case where the condition that $E(\Lambda)$ is a basis for $L^{2}(\mathcal{D})$ is relaxed. In particular, one can ask which measurable sets $\mathcal{D}$ have infinite sets of orthogonal exponentials? This leads us to the main problem of this paper.\\

\textbf{Main Problem}: \textbf{Let} $\mathcal{D}$ \textbf{be a bounded measurable set with a positive Lebesgue measure, then, does} $\mathcal{D}$ \textbf{have an infinite set of orthogonal exponentials} $E(\Lambda)$ \textbf{and if so then} \textbf{how the set} $\Lambda$ \textbf{can be constructed}?\\

In this paper we deal with the above main problem for the case where $\mathcal{D}$ is assumed to be a convex polytope $P$. In Theorem 2.1 we show that every convex, simple and rational polytope has an infinite set of orthogonal exponentials and show how to construct such a set. In Theorem 2.2 we show a similar result for polytopes which are not necessarily simple but which satisfy the condition that their edges are not contained in a certain family of parallel hyperplanes (see Theorem 2.2 for the exact formulation). In Theorem 2.4 we obtain for a given simple, convex polytope $P$ a partial characterization of its infinite orthogonal sets $E(\Lambda)$ of exponentials where $\Lambda$ is a lattice and show that $\Lambda$ must be contained on certain hyperplanes determined by $P$.

The constructions described in Theorems 2.1-2.2 will unfortunately all ways produce orthogonal sets $E(\Lambda)$ of exponentials for which $\Lambda$ has a zero density. The exact definition of a density $D(\Lambda)$ of a subset $\Lambda\subseteq\Bbb R^{d}$ will be given rigourously later in this chapter. The concept of a density $D(\Lambda)$ of a subset $\Lambda\subseteq\Bbb R^{d}$ is important since in some sense it measures how much the orthogonal set $E(\Lambda)$ of exponentials, for a domain $\mathcal{D}$, is "far away" from being a basis for $L^{2}(\mathcal{D})$. Indeed, this is a conclusion from a result obtained by Landau in \cite{La67} which implies in particular that if $E(\Lambda)$ is an orthogonal set of exponentials (with a uniform normalizing factor) for a bounded measurable set $\mathcal{D}$ such that $E(\Lambda)$ is also a basis for $L^{2}(\mathcal{D})$ then $D(\Lambda)\geq |\mathcal{D}| / (2\pi)^{d}$. This implies in particular that the above constructions cannot provide orthogonal bases of exponentials for a given convex polytope $P$.

This leads us to the final result of this paper, obtained in Theorem 2.5, where we show how to obtain infinite sets of orthogonal exponentials $E(\Lambda)$ where $\Lambda$ will have a positive density. The construction will be given for a subset of polytopes which are obtained by orthogonal projections of affine transformations of hypercubes, i.e., zonotopes, but also the obtained set $E(\Lambda)$ will be orthogonal with respect to a positive weight function which results from the projections described above. It can be easily proved that Landau's necessary condition that $D(\Lambda)\geq |\mathcal{D}| / (2\pi)^{d}$ also holds in case where one considers orthogonal exponentials with respect to some weight function. For this we give an explicit lower bound for the density $D(\Lambda)$ of $\Lambda$, which depends on each given zonotope $P$, where in case that this lower bound is greater than $|P| / (2\pi)^{d}$ then one would hope that the above construction will provide an orthogonal basis $E(\Lambda)$ of exponentials for $L^{2}(P)$. However, the possibility of the existence of such a construction, which will provide a basis, is not going to be addressed in this paper.\\

Before formulating the main results we give basic definitions and notations that will be used later on in the text.

Denote by $\Bbb R^{d}$ the $d$ dimensional Euclidean space, by $e_{i}, 1\leq i\leq d$, the standard coordinate unit vectors and by $C_{d} = [-1, 1]^{d}$ the unit cube in $\Bbb R^{d}$. For a subset $\Lambda\subset\Bbb R^{d}$ denote by $E(\Lambda) = \{e^{i\langle \lambda, \cdot\rangle}:\lambda\in\Lambda\}$ its corresponding set of exponentials. The set $\Lambda$ is called a (full rank) \textbf{lattice} if it has the form
$$\hskip-3.25cm\Lambda = M\Bbb Z^{d} = \{(Me_{1})n_{1} + ... + (Me_{d})n_{d}:n_{i}\in\Bbb Z, 1\leq i\leq d\}$$
for some non degenerate matrix $M\in\Bbb R^{d\times d}$. The set $\Lambda$ is said to be a lattice of rank $m, 1\leq m\leq d$, if it is contained in a subspace of $\Bbb R^{d}$ of dimension $m$ and is a lattice with respect to this subspace.

For the set $\Lambda$ define the following quantities
\vskip-0.2cm
\begin{equation}\label{eq:2}\hskip-2cm D^{+}(\Lambda) = \lim_{\rho\rightarrow\infty}\sup_{x\in\Bbb R^{d}}\frac{|\Lambda\cap(x + S_{\rho})|}{\rho^{d}}, D^{-}(\Lambda) = \lim_{\rho\rightarrow\infty}\inf_{x\in\Bbb R^{d}}\frac{|\Lambda\cap(x + S_{\rho})|}{\rho^{d}}\end{equation}
where $S_{\rho} = [0, \rho)^{d}$. Then, if both $D^{\pm}(\Lambda)$ exist (i.e., the corresponding limits exist) and $D^{+}(\Lambda) = D^{-}(\Lambda)$ then the \textbf{density} of $\Lambda$ is defined to be equal to $D^{+}(\Lambda)$ (or $D^{-}(\Lambda)$) and is denoted by $D(\Lambda)$. In case where $\Lambda$ is a lattice determined by the matrix $M$ then it can be easily proved that $D(\Lambda) = 1 / |\det M|$.\\

A \textbf{convex polytope} in $\Bbb R^{d}$ is defined to be the convex hull $P$ of a finite set of points in $\Bbb R^{d}$ such that $P$ has a positive Lebesgue measure. A convex polytope $P\subset\Bbb R^{d}$ is called \textbf{simple} if each of its vertices is connected to exactly $d$ other vertices and is called \textbf{rational} if each of its vertices has only rational components.

A convex polytope $P\subset\Bbb R^{m}$ is called a \textbf{zonotope} of dimension $m\geq 1$ if it is the orthogonal projection of a set of the form $MC_{d}$ on $\Bbb R^{m}$ where $d - 1\geq m$ and where $M$ is a non degenerate $d\times d$ matrix.

For a convex polytope $P$ we denote by $F_{P}$ the \textbf{Fourier transform} of the indicator function of $P$, i.e.,
\vskip-0.2cm
$$\hskip-8.7cm F_{P}(\omega) = \int_{P}e^{-i\langle x, \omega\rangle}dx, \omega\in\Bbb R^{d},\\$$
and denote by $\mathcal{Z}(F_{P})$ the set of zeros of $F_{P}$.

As was mentioned above, the set $E(\Lambda)$ is a set of orthogonal exponentials for the convex polytope $P$ if condition (\ref{eq:1}) is satisfied for $\mathcal{D} = P$. Observe that condition (\ref{eq:1}) is satisfied for $E(\Lambda)$ if and only if it is satisfied for $E(\Lambda - \lambda)$ for any arbitrary $\lambda$ in $\Bbb R^{d}$ that we choose. Hence, we will all ways assume, without loss of generality, that $\bar{0}\in\Lambda$. In terms of the Fourier transform observe that $E(\Lambda)$ is a set of orthogonal exponentials for $P$ if and only if the following condition
\vskip-0.2cm
$$\hskip-9.75cm\Lambda - \Lambda\subseteq\mathcal{Z}(F_{P})\cup\{\bar{0}\}$$
is satisfied. This condition will be used during the proofs of the main results. If $\Lambda$ is a lattice then $\Lambda = \Lambda - \Lambda$ and hence this last condition reduces to the condition that $\Lambda\subseteq\mathcal{Z}(F_{P})\cup\{\bar{0}\}$.\\

For a convex polytope $P$ in $\Bbb R^{d}$ and a subset $\Lambda\subset\Bbb R^{d}$ the set $E(\Lambda)$ is said to be an \textbf{orthogonal set of exponentials} for the polytope $P$ \textbf{with respect to a weight function} $W$ if
\vskip-0.2cm
$$\hskip-1.5cm\int_{P}e^{i\langle \lambda, x\rangle}\overline{e^{i\langle \lambda', x\rangle}}W(x)dx = \int_{P}e^{i\langle\lambda - \lambda', x\rangle}W(x)dx = 0, \forall\lambda, \lambda'\in\Lambda, \lambda\neq\lambda'.$$\\
Before formulating and proving the main results we would like to add one last remark which will be used during the text. If $P$ and $P'$ are two affinely equivalent convex polytopes in $\Bbb R^{d}$, i.e., there exist a non degenerate matrix $M\in\Bbb R^{d\times d}$ and a vector $v\in\Bbb R^{d}$ such that $P' = MP + v$, then $P$ has an infinite set of orthogonal exponentials if and only if the same is true for $P'$. This can be easily observed from equation (\ref{eq:1}) by making the change of variables $x = M^{-1}(y - v)$.

\section{The Main Results}

Theorems 2.1, 2.2 and 2.4 deal with the existence and construction of infinite sets of orthogonal exponentials for convex polytopes without considering weight functions. The first main result of this paper, Theorem 2.1, guarantees the existence of infinite sets of orthogonal exponentials for simple, convex polytopes which are also rational:

\begin{thm}

Every simple, convex and rational polytope $P$ has an infinite set of orthogonal exponentials.

\end{thm}

In the proof of Theorem 2.1 we also give an explicit construction of the set $\Lambda$ for which $E(\Lambda)$ is an infinite set of orthogonal exponentials for $P$.

The second main result of this paper, Theorem 2.2, guarantees the existence of an infinite set of orthogonal exponentials for polytopes which are not necessarily simple and with the requirement that only one component in the vertices of each polytope $P$ in consideration will be rational. However, this also comes with the additional requirement that the edges of $P$ must be "in general position" in some sense that will be rigourously defined below. The exact formulation of Theorem 2.2 is as follows.

\begin{thm}

Let $P$ be a convex polytope and assume that for some $k$, $1\leq k\leq d$, the $k^{th}$ component in each vertex of $P$ is rational and that non of the edges of $P$ are contained in a hyperplane of the form $x_{k} = c$ where $c\in\Bbb R$. Then, $P$ has an infinite set of orthogonal exponentials.

\end{thm}

Theorem 2.2 will be proved by showing that the Fourier transform $F_{P}$, of the indicator function of $P$, vanishes on $\Lambda\setminus\{\bar{0}\}$ where $\Lambda$ is a lattice of rank one which is contained on the $k^{th}$ axis. Then, as was explained in the introduction, this will immediately imply that $E(\Lambda)$ forms an infinite set of orthogonal exponentials for $P$.

\begin{remark}

As was explained in the introduction, if $P$ and $P'$ are two affinely equivalent convex polytopes then $P$ has an infinite set of orthogonal exponentials if and only if the same is true for $P'$. Hence, Theorems 2.1 and 2.2 are actually true for a broader family of convex polytopes which are affinely equivalent to the ones which appear in their formulations. For example, every simple, convex polytope $P$ whose vertices lie on a full rank lattice $\Lambda$ has an infinite set of orthogonal exponentials. Indeed, obviously there exists an affine transformation which maps $\Lambda$ to $\Bbb Z^{d}$ and thus it will map $P$ to a polytope $P'$ whose vertices lie on $\Bbb Z^{d}$. Since any linear transformation preserves the convexity and simplicity properties of polytopes it follows that $P'$ is simple, convex and rational and thus has an infinite set of orthogonal exponentials which implies that the same is true for $P$.

\end{remark}

Theorem 2.2 guarantees that every polytope $P$, satisfying the conditions described in the theorem, has an infinite set $E(\Lambda)$ of orthogonal exponentials where $\Lambda$ is a lattice. In general, one would like to characterize the set of convex polytopes which have infinite sets $E(\Lambda)$ of orthogonal exponentials for some lattice $\Lambda$. While a complete characterization of such polytopes is not obtained in this paper, we give a necessary condition for the above property to hold for simple polytopes. The exact formulation is given in Theorem 2.4 below.

\begin{thm}

Let $P$ be a simple, convex polytope with an infinite set $E(\Lambda)$ of orthogonal exponentials and assume that $\Lambda$ is a (not necessarily a full rank) lattice. Then, for every $\omega_{0}$ in $\Lambda\setminus\{\bar{0}\}$ there exist two distinct vertices $v$ and $v'$ of $P$ and an integer $m$ such that $\langle \omega_{0}, v - v'\rangle = 2\pi m$.

\end{thm}

Geometrically, Theorem 2.4 implies that each point in $\Lambda\setminus\{\bar{0}\}$ must be contained on a prescribed set of hyperplanes determined by the vertices of $P$. The necessary condition in Theorem 2.4 does not take the rank of the lattice $\Lambda$ into consideration which suggests that this condition might not be sharp (i.e., not a sufficient condition). It will be interesting to obtain a stronger version of Theorem 2.4 which involves the rank of $\Lambda$ as well (i.e., the higher the rank of $\Lambda$ the more conditions are imposed in the necessary condition).

Theorems 2.1 and 2.2 guarantee the existence of an infinite set of orthogonal exponentials $E(\Lambda)$ under some conditions on the polytopes in consideration. However, if one also considers the density of $\Lambda$ then unfortunately the constructions given in these theorems all ways produce sets $\Lambda$ of zero density. Indeed, in the proof of Theorem 2.2 the set $\Lambda$ being constructed is a lattice of rank one which obviously has a zero density. The proof that the general construction presented in the proof of Theorem 2.1 also produces sets $\Lambda$ of zero density is more complicated and is given in Proposition 4.1 in the Appendix.

Hence, our next aim is to try to find families $\Omega$ of polytopes for which one can construct an infinite set of orthogonal exponentials $E(\Lambda)$ where $\Lambda$ has a positive density and to try to estimate the density $D(\Lambda)$ of $\Lambda$. In this paper we show how to obtain such sets $\Lambda$ for zonotopes where the set $E(\Lambda)$ will be orthogonal with respect to a weight function which depends on each zonotope in consideration. As was mentioned in the introduction, for a positive integer $m$, a zonotope $P$ of dimension $m$ is an orthogonal projection of a set of the form $MC_{d}$ on $\Bbb R^{m}$ where $d - 1\geq m$ and where $M$ is a non degenerate $d\times d$ matrix. In order to guarantee the construction of the set $E(\Lambda)$ of orthogonal exponentials one has to assume non trivial conditions on the matrix $M$ and hence the family $\Omega$ of polytopes, for which the construction of the set $E(\Lambda)$ of orthogonal exponentials can be applied, is strictly contained in the set $\Omega'$ of all zonotopes. However, since the construction can be applied whenever $M$ is rational and non degenerate it follows that $\Omega$ is dense in $\Omega'$.

The exact formulation of the main result mentioned above is given in Theorem 2.5 below. Before formulating Theorem 2.5 we use the notation $\mathbf{H}_{x_{1},...,x_{m}}, 1\leq m\leq d - 1$ for the $d - m$ dimensional plane in $\Bbb R^{d}$ given as
\vskip-0.2cm
$$\hskip-6cm \mathbf{H}_{x_{1},...,x_{m}} = \{u\in\Bbb R^{d}:u_{1} = x_{1},...,u_{m} = x_{m}\}.$$

\begin{thm}

Let $d$ and $m$ be positive integers such that $m\leq d - 1$ and let $M$ be a nondegenerate matrix in $\Bbb R^{d\times d}$ which satisfies the condition that there exist $m$ linearly independent vectors $\bar{k}_{1},...,\bar{k}_{m}\in\Bbb Z^{d}\setminus\{\bar{0}\}$ such that
\vskip-0.2cm
\begin{equation}\label{eq:20}\hskip-4.5cm\langle \bar{k}_{j}, M^{-1}e_{i}\rangle = 0, i = m + 1, m + 2,...,d, j = 1,...,m.\end{equation}
Then, the orthogonal projection of $MC_{d}$ to $\Bbb R^{m}$ has an infinite set of orthogonal exponentials $E(\Lambda)$, with respect to a weight function $W_{M, m}$, where $\Lambda$ is given by
$$\hskip0.8cm\Lambda = \left\{\pi\left(\langle \bar{k}, M^{-1}e_{1}\rangle,...,\langle \bar{k}, M^{-1}e_{m}\rangle\right):\bar{k}\in\Bbb Z^{d}, \langle \bar{k}, M^{-1}e_{i}\rangle = 0, i = m + 1, m + 2,...,d\right\}$$
and satisfies $D(\Lambda) \geq \pi^{-d}|\det M| / |\det U|$ where
$$\hskip-4.1cm U = \left(\begin{array}{c}
K\\
e_{m + 1}M\\
...\\
e_{d}M
\end{array}\right)\hskip0.1cm\textit{and}\hskip0.1cm\textit{where}\hskip0.2cm K = \left(\begin{array}{c}
\bar{k}_{1}\\
...\\
\bar{k}_{m}
\end{array}\right)\in\Bbb Z^{m\times d}.$$
The weight function $W_{M, m}$ is given by
\vskip-0.2cm
$$\hskip-3.9cm W_{M,m}(x_{1},...,x_{m}) = \int_{MC_{d}\cap \mathbf{H}_{x_{1},...,x_{m}}}dx_{m + 1}dx_{m + 2}...dx_{d}.$$

\end{thm}

In Lemma 4.2 in the Appendix we give necessary and sufficient conditions on the matrix $M$ so that condition (\ref{eq:20}) is satisfied. In particular, Lemma 4.2 guarantees that condition (\ref{eq:20}) is satisfied if $M$ is rational and for the case where $m = d - 1$ then condition (\ref{eq:20}) is satisfied if and only if the last column of $M^{-1}$ is proportional to an integer vector.

\section{The Proofs of the Main Results and Related Examples}

The proofs of Theorems 2.1 and 2.4 are based on the following explicit formula (\ref{eq:92}) which expresses the Fourier transform of the indicator function of a simple convex polytope $P$ via its vertices. Formula (\ref{eq:92}) is obtained by considering the result obtained in \cite{Gravin, Lawrence} for the calculation of the moments of a simple polytope $P$ in $\Bbb R^{d}$:
\vskip-0.2cm
\begin{equation}\label{eq:90}\hskip-2.5cm\int_{P}\langle x,\omega\rangle^{j}dx = \frac{j!(-1)^{d}}{(j + d)!}\sum_{v\in\mathrm{Vert}(P)}\langle v, \omega\rangle^{j + d}D_{v}(z), j\in\Bbb N\cup\{0\},\end{equation}
where
\vskip-0.2cm
$$\hskip-7.65cm D_{v}(\omega) = \frac{|\det(\xi_{1}(v),...,\xi_{d}(v))|}{\langle \xi_{1}(v), \omega\rangle...\langle \xi_{d}(v), \omega\rangle}$$
and where $\xi_{1}(v),...,\xi_{d}(v)$ are the edge vectors emanating from the vertex $v$ (i.e., if $v_{1},...,v_{d}$ are the adjacent vertices to $v$ then one can choose $\xi_{i}(v) = v_{i} - v, i = 1,...,d$). Formula (\ref{eq:90}) is true for every $\omega$ in $\Bbb R^{d}$ for which the denominator of the function $D_{v}$, for each vertex $v$, does not vanish. Hence, using formula (\ref{eq:90}), the Taylor expansion of the exponential function $x\mapsto\exp(x)$ for the function $\omega\mapsto \exp(-i\langle\cdot, \omega\rangle)$ and using also the following companion identities
\vskip-0.2cm
\begin{equation}\label{eq:91}\hskip-7cm\sum_{v\in\textrm{Vert}(P)}\langle v, \omega\rangle^{j}D_{v}(\omega) = 0, 0\leq j\leq d - 1\end{equation}
(see \cite[Eq. 3]{Gravin}) we have the following corollary.

\begin{cor}

Let $P$ be a simple convex polytope in $\Bbb R^{d}$ and let $F_{P}$ be the Fourier transform of the indicator function of $P$. Then, for every point $\omega$ which is not orthogonal to any of the edges of $P$ we have that
\vskip-0.2cm
\begin{equation}\label{eq:92}\hskip-7.5cm F_{P}(\omega) = i^{-d}\sum_{v\in\mathrm{Vert}(P)}D_{v}(\omega)e^{-i\langle v, \omega\rangle}.\end{equation}

\end{cor}

\textbf{Proof of Theorem 2.1}: Let $F_{P}$ be the Fourier transform of the indicator function of $P$. Using formula (\ref{eq:92}), for every point $\omega$ which is not orthogonal to any of the edges of $P$ we have
\vskip-0.2cm
$$\hskip-7.4cm F_{P}(\omega) = i^{-d}\sum_{v\in\mathrm{Vert}(P)}D_{v}(\omega)e^{-i\langle v, \omega\rangle}.$$
Since $P$ is rational it follows that there exists a positive integer $n$ such that every vertex $v$ in $\mathrm{Vert}(P)$ can be written as
\vskip-0.2cm
$$\hskip-5.2cm v = \left(\frac{m_{1}^{v}}{n},...,\frac{m_{d}^{v}}{n}\right)\hskip0.2cm\mathrm{where}\hskip0.3cm m_{i}^{v}\in\Bbb Z, 1\leq i\leq d.$$
Hence, for every point $\omega$ of the form $\omega = 2\pi n(k_{1},...,k_{d})$, where $k_{i}\in\Bbb Z, 1\leq i\leq d$,
such that $\omega$ is not orthogonal to any of the edges of $P$, we have
$$F_{P}(2\pi nk_{1},...,2\pi nk_{d}) = i^{-d}\sum_{v\in\mathrm{Vert}(P)}D_{v}(2\pi nk_{1},...,2\pi nk_{d})e^{-2\pi i(k_{1}m_{1}^{v} + ... + k_{d}m_{d}^{v})}$$
\begin{equation}\label{eq:100}\hskip1.3cm = i^{-d}\sum_{v\in\mathrm{Vert}(P)}D_{v}(2\pi nk_{1},...,2\pi nk_{d}) = 0\end{equation}
where in the last passage we used identity (\ref{eq:91}) for $j = 0$. Let us denote by $H_{1},...,H_{N}$ the set of hyperplanes, which pass through the origin in $\Bbb R^{d}$, on which the denominators of the functions $D_{v} = D_{v}(\omega)$ ($v\in\mathrm{Vert}(P)$) vanish and let us define the following lattice
\vskip-0.2cm
$$\hskip-3.1cm L = \{(2\pi nk_{1},...,2\pi nk_{d}):k_{i}\in\Bbb Z, 1\leq i\leq d\} = 2\pi n\cdot \Bbb Z^{d}.$$

\begin{figure}[t]

    \caption{An illustration of the construction of the set $\Lambda$ for the convex polygon in the plane with vertices $(\pm1, 0), (\pm2, 1)$. In this case the lattice $L$ consists of the blue points, $N = 3$ and the lines $H_{1}, H_{2}$ and $H_{3}$ are given respectively by $\omega_{1} = 0, \omega_{1} - \omega_{2} = 0, \omega_{1} + \omega_{2} = 0$. Observe that each time we choose a point $p$ from $L$ we have to omit all the points in $L$ which are contained on the lines which are parallel to $H_{1}, H_{2}$ and $H_{3}$ and pass through $p$ (except of course the point $p$).}

    \begin{tikzpicture}

        \draw [<->] (-5, 0) -- (5, 0);
        \draw [<->] (0, -5) -- (0, 5);
        \draw [dashed, line width = 1.5, red] (0, -4.8) -- (0, 4.8);
        \draw [dashed, line width = 1, red] (-4.8, -4.8) -- (4.8, 4.8);
        \draw [dashed, line width = 1, red] (-4.8, 4.8) -- (4.8, -4.8);

        \draw [dashed, line width = 1, orange] (3, -4.8) -- (3, 4.8);
        \draw [dashed, line width = 1, orange] (-3, -5) -- (5, 3);
        \draw [dashed, line width = 1, orange] (- 1, 5) -- (5, -1);

        \draw [fill = blue] (0, 0) circle [radius = 0.05];
        \draw [fill = blue] (1, 1) circle [radius = 0.05];
        \draw [fill = blue] (- 1, - 1) circle [radius = 0.05];
        \draw [fill = blue] (2, 2) circle [radius = 0.05];
        \draw [fill = blue] (- 2, - 2) circle [radius = 0.05];
        \draw [fill = blue] (3, 3) circle [radius = 0.05];
        \draw [fill = blue] (- 3, - 3) circle [radius = 0.05];
        \draw [fill = blue] (4, 4) circle [radius = 0.05];
        \draw [fill = blue] (- 4, - 4) circle [radius = 0.05];

        \draw [fill = blue] (0, 2) circle [radius = 0.05];
        \draw [fill = blue] (1, 3) circle [radius = 0.05];
        \draw [fill = blue] (- 1, 1) circle [radius = 0.05];
        \draw [fill = blue] (2, 4) circle [radius = 0.05];
        \draw [fill = blue] (- 2, 0) circle [radius = 0.05];
        \draw [fill = blue] (3, 5) circle [radius = 0.05];
        \draw [fill = blue] (- 3, - 1) circle [radius = 0.05];
        \draw [fill = blue] (- 4, - 2) circle [radius = 0.05];

        \draw [fill = blue] (0, 4) circle [radius = 0.05];
        \draw [fill = blue] (1, 5) circle [radius = 0.05];
        \draw [fill = blue] (- 1, 3) circle [radius = 0.05];
        \draw [fill = blue] (- 2, 2) circle [radius = 0.05];
        \draw [fill = blue] (- 3, 1) circle [radius = 0.05];
        \draw [fill = blue] (- 4, 0) circle [radius = 0.05];

        \draw [fill = blue] (- 1, 5) circle [radius = 0.05];
        \draw [fill = blue] (- 2, 4) circle [radius = 0.05];
        \draw [fill = blue] (- 3, 3) circle [radius = 0.05];
        \draw [fill = blue] (- 4, 2) circle [radius = 0.05];

        \draw [fill = blue] (- 3, 5) circle [radius = 0.05];
        \draw [fill = blue] (- 4, 4) circle [radius = 0.05];

        \draw [fill = blue] (0, -2) circle [radius = 0.05];
        \draw [fill = blue] (1, - 1) circle [radius = 0.05];
        \draw [fill = blue] (- 1, - 3) circle [radius = 0.05];
        \draw [fill = blue] (2, 0) circle [radius = 0.05];
        \draw [fill = blue] (- 2, - 4) circle [radius = 0.05];
        \draw [fill = blue] (3, 1) circle [radius = 0.05];
        \draw [fill = blue] (- 3, - 5) circle [radius = 0.05];
        \draw [fill = blue] (4, 2) circle [radius = 0.05];

        \draw [fill = blue] (0, -4) circle [radius = 0.05];
        \draw [fill = blue] (1, - 3) circle [radius = 0.05];
        \draw [fill = blue] (- 1, - 5) circle [radius = 0.05];
        \draw [fill = blue] (2, -2) circle [radius = 0.05];
        \draw [fill = blue] (3, - 1) circle [radius = 0.05];
        \draw [fill = blue] (4, 0) circle [radius = 0.05];

        \draw [fill = blue] (1, - 5) circle [radius = 0.05];
        \draw [fill = blue] (2, -4) circle [radius = 0.05];
        \draw [fill = blue] (3, - 3) circle [radius = 0.05];
        \draw [fill = blue] (4, -2) circle [radius = 0.05];

        \draw [fill = blue] (3, - 5) circle [radius = 0.05];
        \draw [fill = blue] (4, - 4) circle [radius = 0.05];

    \end{tikzpicture}

\end{figure}
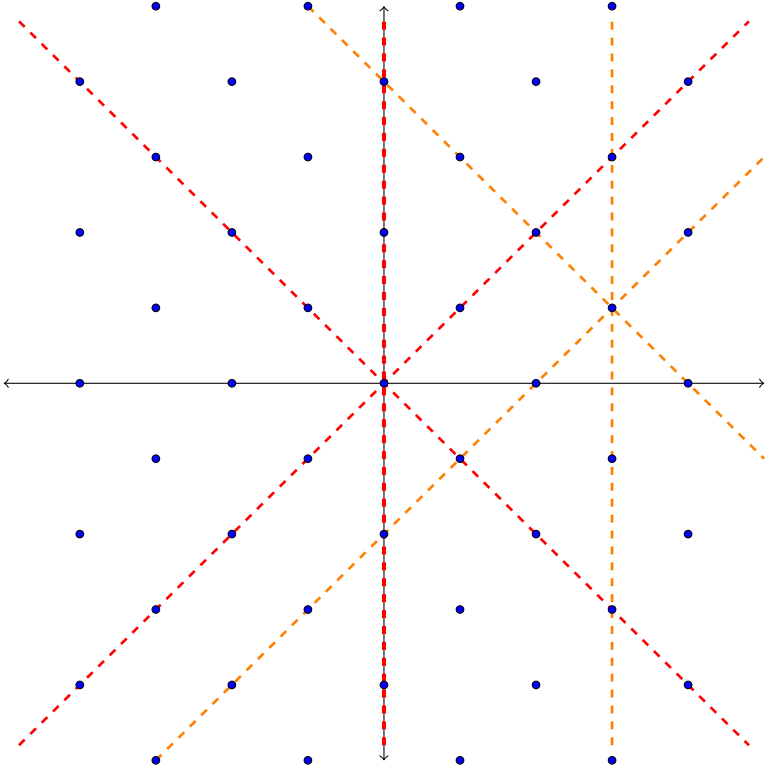
\vskip-0.01cm
\hskip-0.625cm Then, from equation (\ref{eq:100}) it follows that
\vskip-0.2cm
\begin{equation}\label{eq:101}\hskip-7.8cm L\setminus(H_{1}\cup...\cup H_{N})\subseteq\mathcal{Z}(F_{p}).\end{equation}
Now, let us construct the set $E(\Lambda)$ of orthogonal exponentials for $P$. Let us assume, without loss of generality, that $\bar{0}\in\Lambda$. Then, from equation (\ref{eq:101}) it follows that if we can find an infinite subset $\Lambda$ such that $\bar{0}\in \Lambda$ and $\Lambda$ satisfies the condition that
\vskip-0.2cm
\begin{equation}\label{eq:102}\hskip-4cm \Lambda - \Lambda\subset \left[L\setminus(H_{1}\cup...\cup H_{N})\right]\cup\{\bar{0}\}\subseteq\mathcal{Z}(F_{p})\cup\{\bar{0}\}\end{equation}
then $E(\Lambda)$ will be an infinite set of orthogonal exponentials for $P$. Let us write
$\Lambda = \{\bar{0}, x_{1}, x_{2},...\}$ where our aim is to construct the points $x_{i}\in\Lambda, i\geq 1$ such that condition (\ref{eq:102}) is satisfied. Since $\bar{0}\in \Lambda$ the point $x_{1}$ must be chosen from the set
\vskip-0.2cm
$$\hskip-8.5cm \Lambda_{1} = L\setminus(H_{1}\cup...\cup H_{N}).$$
Observe that the set $\Lambda_{1}$ must be infinite since $L$ is a lattice and then we omit only a finite set of hyperplanes. Hence, there exists a point $x_{1}\in \Lambda_{1}$. Next, since $\bar{0}$ and $x_{1}$ are in $L$ then in order for condition (\ref{eq:102}) to be satisfied the point $x_{2}$ must be chosen from the set
\vskip-0.2cm
$$\hskip-6cm \Lambda_{2} = \Lambda_{1}\setminus((x_{1} + H_{1})\cup...\cup(x_{1} + H_{N}))$$
$$\hskip-2.55cm = L\setminus(H_{1}\cup...\cup H_{N}\cup(x_{1} + H_{1})\cup...\cup(x_{1} + H_{N})).$$
Again, since we omit only a finite set of hyperplanes from $L$ it follows that $\Lambda_{2}$ must be infinite and thus there exists a point $x_{2}\in \Lambda_{2}$. Continuing in this way we can construct the whole set $\Lambda$. This proves Theorem 2.1.$\square$\\


\textbf{Proof of Theorem 2.2}: Before proving Theorem 2.2 let us introduce some notation. For a $d$ dimensional polytope $P$ we denote by $\mathbf{Face}(P)$ the set of all its facets (i.e., the faces of $P$ of dimension $d - 1$). For $1\leq k\leq d$ denote by $x_{k}^{\perp}$ the complementary space to the space spanned by $e_{k}$ and denote by $\mathbf{Proj}_{k}$ the orthogonal projection on the $k^{th}$ axis, i.e., $\mathbf{Proj}_{k}(x_{1},...,x_{d}) = x_{k}$. For a set of $k$ indices $1\leq i_{1} < ... < i_{k}\leq d$ denote by $\mathbf{Proj}^{\perp}_{i_{1},...,i_{k}}$ the orthogonal projection to the complementary space $X^{\perp}$ of $X$ where $X$ is the linear space spanned by $e_{i_{1}},...,e_{i_{k}}$, i.e.,
$$\mathbf{Proj}^{\perp}_{i_{1},...,i_{k}}(x_{1},...,x_{d}) = (x_{1},...,x_{i_{1} - 1},x_{i_{1} + 1},...,x_{i_{2} - 1},x_{i_{2} + 1},.......,x_{i_{k} + 1},...,x_{d}).$$
Define $d\hat{x}_{i_{1},...,i_{k}}$ to be the infinitesimal volume measure obtained from $dx_{i}, 1\leq i\leq d,$ where the terms $dx_{i_{1}},...,dx_{i_{k}}$ are omitted, i.e.,
$$\hskip-3cm d\hat{x}_{i_{1},...,i_{k}} = dx_{1}...dx_{i_{1} - 1}dx_{i_{1} + 1}...dx_{i_{2} - 1}dx_{i_{2} + 1}......dx_{i_{k} + 1}...dx_{d}.$$
Lastly, observe that there exists a unique exterior unit normal $\hat{n}$ for each point $p$ which is contained in the relative interior of a facet $\mathbf{F}$ of a convex polytope $P$. This normal $\hat{n}$ does not change if one varies the point $p$ in the relative interior of $\mathbf{F}$ and hence $\hat{n}$ depends only on $\mathbf{F}$ and thus we can denote it by $\nu(\mathbf{F})$. We denote by $\nu_{1}(\mathbf{F})$,...,$\nu_{d}(\mathbf{F})$ its components. With this notation we are ready to prove Theorem 2.2.\\

Let us assume, without loss of generality, that $k = 1$. We can further assume that the first component of each vertex of $P$ is an integer. Indeed, if $N$ is a common denominator of all the rational first components in all the vertices of $P$ then the affine transformation $\varphi:x\mapsto Nx$ will map the convex polytope $P$ to a convex polytope $P'$ whose vertices have the property that their first component is an integer. Now, we can use the fact that the property of having an infinite set of orthogonal exponentials is invariant with respect to affine transformations and that the property that non of the edges of $P$ are contained in a hyperplane of the form $x_{1} = c$ is preserved under the map $\varphi$. Also, observe that the assumption on the edges of $P$ is equivalent to the assumption that non of the faces of $P$ of dimension greater or equal to $1$ are contained in a hyperplane of the form $x_{1} = c$.\\

From the divergence formula we have, for $\omega_{1}\neq0$, that
\begin{equation}\label{eq:110}\hskip-4cm\int_{P}e^{-i\langle \omega, x\rangle}dx = -\frac{1}{i\omega_{1}}\sum_{\mathbf{F}\in\mathbf{Face}(P)}\nu_{1}(\mathbf{F})\int_{\mathbf{F}}e^{-i\langle \omega, x\rangle}d\mathbf{F}(x)\end{equation}
where $d\mathbf{F}$ is the standard infinitesimal volume measure on $\mathbf{F}$. Let $F_{P}$ denote the Fourier transform of the indicator function of $P$. Then, assuming that $\omega_{2} = ... = \omega_{d} = 0$ and using formula (\ref{eq:110}) we have
\begin{equation}\label{eq:111}\hskip-1.5cm F_{P}(\omega_{1}, 0..., 0) = \int_{P}e^{-ix_{1}\omega_{1}} dx = -\frac{1}{i\omega_{1}}\sum_{\mathbf{F}\in\mathbf{Face}(P)}\nu_{1}(\mathbf{F})\int_{\mathbf{F}}e^{-ix_{1}\omega_{1}}d\mathbf{F}(x).\end{equation}
Since, by assumption, non of the facets of $P$ are contained in a hyperplane of the form $x_{1} = c$ it follows that for each facet $\mathbf{F}$ there exists an integer $k = k_{\mathbf{F}}$, $2\leq k\leq d$, such that $\nu_{k_{\mathbf{F}}}(\mathbf{F})\neq 0$. Hence, $\mathbf{F}$ can be parameterized by $x_{1},...,x_{k_{\mathbf{F}} - 1}, x_{k_{\mathbf{F}} + 1},...,x_{d}$ and the infinitesimal volume measure $d\mathbf{F}$ of $\mathbf{F}$ is given in these coordinates by $d\mathbf{F}(x) = d\hat{x}_{k_{\mathbf{F}}} / |\nu_{k_{\mathbf{F}}}(\mathbf{F})|$. Thus, from equation (\ref{eq:111}) we have
\vskip-0.2cm
\begin{equation}\label{eq:112}\hskip-2.6cm F_{P}(\omega_{1}, 0..., 0) = -\frac{1}{i\omega_{1}}\sum_{\mathbf{F}\in\mathbf{Face}(P)}\frac{\nu_{1}(\mathbf{F})}{|\nu_{k_{\mathbf{F}}}(\mathbf{F})|}
\int_{\textbf{Proj}^{\perp}_{k_{\mathbf{F}}}(\mathbf{F})}e^{-ix_{1}\omega_{1}}d\hat{x}_{k_{\mathbf{F}}}.\end{equation}
Using the divergence formula again on the right hand side of equation (\ref{eq:112}) we have
\begin{equation}\label{eq:113}F_{P}(\omega_{1}, 0..., 0) = \left(- \frac{1}{i\omega_{1}}\right)^{2}\sum_{\mathbf{F}\in\mathbf{Face}(P)}
\sum_{\mathbf{F}'\in\mathbf{Face}(\mathbf{Proj}^{\perp}_{k_{\mathbf{F}}}(\mathbf{F}))}
\frac{\nu_{1}(\mathbf{F})\nu_{1}(\mathbf{F}')}{|\nu_{k_{\mathbf{F}}}(\mathbf{F})|}
\int_{\mathbf{F}'}e^{-ix_{1}\omega_{1}}d\mathbf{F}'(x).\end{equation}
Now, observe that by our main assumption it follows that none of the subfaces, of each face $\mathbf{F}$, are contained in a hyperplane of the form $x_{1} = c$ and thus the same is true for its projection on the subspace $x_{k_{\mathbf{F}}}^{\perp}$. That is, non of the subfaces $\mathbf{F}'$ of the projection of $\mathbf{F}$ to $x_{k_{\mathbf{F}}}^{\perp}$ are contained in a hyperplane of the form $\Pi:x_{1} = c$ where now we think of $\Pi$ as a hyperplane in the subspace $x_{k_{\mathbf{F}}}^{\perp}$. Hence, for each such subface $\mathbf{F}'$, which is also a facet of the projection of $\mathbf{F}$, there exists an integer $k_{\mathbf{F}'}, 2\leq k_{\mathbf{F}'}\leq d$, such that $\nu_{k_{\mathbf{F}'}}(\mathbf{F}')\neq0$ (where now we think of $\nu(\mathbf{F}')$ as the normal of the projection of $\mathbf{F}$ at $\mathbf{F}'$ in the subspace $x_{k_{\mathbf{F}}}^{\perp}$). Hence, exactly as was explained before we obtain from equation (\ref{eq:113}) that
\vskip-0.2cm
$$\hskip-1.5cm F_{P}(\omega_{1}, 0..., 0) = \left(- \frac{1}{i\omega_{1}}\right)^{2}\sum_{\mathbf{F}\in\mathbf{Face}(P)}
\sum_{\mathbf{F}'\in\mathbf{Face}(\mathbf{Proj}^{\perp}_{k_{\mathbf{F}}}(\mathbf{F}))}
\frac{\nu_{1}(\mathbf{F})\nu_{1}(\mathbf{F}')}{|\nu_{k_{\mathbf{F}}}(\mathbf{F})||\nu_{k_{\mathbf{F}'}}(\mathbf{F}')|}$$
$$\hskip5cm\times\int_{\mathbf{Proj}^{\perp}_{k_{\mathbf{F}}, k_{\mathbf{F}'}}(\mathbf{F}')}e^{-ix_{1}\omega_{1}}d\hat{x}_{k_{\mathbf{F}}, k_{\mathbf{F}'}}.$$
Continuing in this way we can obtain that
\vskip-0.2cm
$$\hskip-12.5cm F_{P}(\omega_{1}, 0..., 0)$$
$$\hskip-0.8cm = \left(- \frac{1}{i\omega_{1}}\right)^{d - 1}\sum_{\mathbf{F}\in\mathbf{Face}(P)}
\sum_{\mathbf{F}'\in\mathbf{Face}(\textrm{Proj}^{\perp}_{k_{\mathbf{F}}}(\mathbf{F}))}......
\sum_{\mathbf{F}^{(d - 2)}\in\mathbf{Face}(\mathbf{Proj}^{\perp}_{k_{\mathbf{F}}, k_{\mathbf{F}'},...,k_{\mathbf{F}^{(d - 3)}}}(\mathbf{F}^{(d - 3)}))}$$
\begin{equation}\label{eq:114}\hskip0.2cm\frac{\nu_{1}(\mathbf{F})\nu_{1}(\mathbf{F}')...\nu_{1}(\mathbf{F}^{(d - 2)})}{|\nu_{k_{\mathbf{F}}}(\mathbf{F})||\nu_{k_{\mathbf{F}'}}(\mathbf{F}')|...|\nu_{k_{\mathbf{F}^{(d - 2)}}}(\mathbf{F}^{(d - 2)})|}
\int_{\mathbf{Proj}^{\perp}_{k_{\mathbf{F}}, k_{\mathbf{F}'},...,k_{\mathbf{F}^{(d - 2)}}}(\mathbf{F}^{(d - 2)})}e^{-ix_{1}\omega_{1}}d\hat{x}_{k_{\mathbf{F}}, k_{\mathbf{F}'},...,k_{\mathbf{F}^{(d - 2)}}}.\end{equation}
Observe that for each chain of subfaces $\mathbf{F}^{(d - 2)}\subset \mathbf{F}^{(d - 3)}\subset...\subset \mathbf{F}'\subset\mathbf{F}$ we have that all the indices $k_{\mathbf{F}},...,k_{\mathbf{F}^{(d - 2)}}$ are different (since each time we choose a different projection). Hence, it follows that $d\hat{x}_{k_{\mathbf{F}}, k_{\mathbf{F}'},...,k_{\mathbf{F}^{(d - 2)}}} = dx_{1}$ and that the function
\vskip-0.2cm
$$\hskip-7.3cm x\mapsto\mathbf{Proj}^{\perp}_{k_{\mathbf{F}}, k_{\mathbf{F}'},...,k_{\mathbf{F}^{(d - 2)}}}(x), x\in\Bbb R^{d}$$
coincides with the orthogonal projection on the $x_{1}$ axis. Since a projection of a connected set is connected it follows that the integration in the right hand side of equation (\ref{eq:114}) must be on a segment on the $x_{1}$ axis with, say, left endpoint $x_{\mathbf{F}^{(d - 2)}}$ and right endpoint $y_{\mathbf{F}^{(d - 2)}}$ which depend on each face $\mathbf{F}^{(d - 2)}$. Also, these end points must be different since otherwise it would have followed that $\mathbf{F}^{(d - 2)}$ is contained in a hyperplane of the form $x_{1} = c$. However, this is impossible since by our construction non of the faces which appear in the index set of the sums in the right hand side of equation (\ref{eq:114}) are contained on such a plane. Hence, we obtain that each integral in the right hand side of equation (\ref{eq:114}) is given by
\vskip-0.2cm
$$\hskip-1.25cm\int_{\mathbf{Proj}^{\perp}_{k_{\mathbf{F}}, k_{\mathbf{F}'},...,k_{\mathbf{F}^{(d - 2)}}}(\mathbf{F}^{d - 2})}e^{-ix_{1}\omega_{1}}d\hat{x}_{k_{\mathbf{F}}, k_{\mathbf{F}'},...,k_{\mathbf{F}^{(d - 2)}}} = \int_{\mathbf{Proj}_{1}(\mathbf{F}^{(d - 2)})}e^{-x_{1}\omega_{1}}dx_{1}$$
$$\hskip-5.15cm = \int_{x_{\mathbf{F}^{(d - 2)}}}^{y_{\mathbf{F}^{(d - 2)}}}e^{-ix_{1}\omega_{1}}dx_{1} = \frac{e^{-i\omega_{1}y_{\mathbf{F}^{(d - 2)}}} - e^{-i\omega_{1}x_{\mathbf{F}^{(d - 2)}}}}{-i\omega_{1}}.$$
Observe that the end points $x_{\mathbf{F}^{(d - 2)}}$ and $y_{\mathbf{F}^{(d - 2)}}$ must be integers from our assumption that the first component of each vertex of $P$ is an integer and the fact that the set of vertices of an orthogonal projection of a polytope is contained in the set of projections of the vertices of this polytope. Hence, it follows that $F_{P}(2\pi k, 0,..., 0) = 0$ for every $k\in\Bbb Z\setminus\{0\}$ which implies that $\mathcal{Z}(F_{P})\cup\{\bar{0}\}$ contains a lattice of rank one. Thus, $P$ has an infinite set of orthogonal exponentials.$\square$\\


\textbf{Proof of Theorem 2.4}: First, observe that for a given lattice $\Lambda$ the fact that the simple, convex polytope $P$ has the infinite set $E(\Lambda)$ of orthogonal exponentials is equivalent to the fact that the Fourier transform $F_{P}$, of the indicator function of $P$, vanishes on $\Lambda\setminus\{\bar{0}\}$. Now, for a given point $\omega_{0}$ in $\Lambda\setminus\{\bar{0}\}$ let us distinguish between the following two cases.

If $\omega_{0}$ is orthogonal to one of the edges of $P$, or equivalently the denominator in one of the functions $D_{v}, v\in\mathrm{Vert}(P)$ vanishes, then there exist two different vertices $v, v'$ of $P$ such that $\langle v - v', \omega_{0}\rangle = 0$ and hence Theorem 2.4 follows for $m = 0$.

Hence, let us assume that $\omega_{0}$ is not orthogonal to any of the edges of $P$. Observe that if $\omega_{0}\in\Lambda\setminus\{0\}$ then $r\omega_{0}\in\Lambda\setminus\{0\}$ for every $r\in\Bbb Z\setminus\{0\}$. Hence, since $F_{P}$ vanishes on $\Lambda\setminus\{0\}$ it follows that it vanishes on $r\omega_{0}, r\in\Bbb Z\setminus\{0\}$. Also, the function $D_{v}$ is homogenous of degree $- d$. Thus, from equation (\ref{eq:92}) it follows that
\vskip-0.2cm
\begin{equation}\label{eq:120}\hskip-7cm\sum_{v\in\mathrm{Vert}(P)}D_{v}(\omega_{0})e^{-ir\langle v, \omega_{0}\rangle} = 0, r\in\Bbb Z\setminus\{0\}.\end{equation}
Observe that by using identity (\ref{eq:91}) for $j = 0$ it follows that equation (\ref{eq:120}) is also true for $r = 0$. Hence, if $v_{1},...,v_{n}$ are the vertices of $P$ then using equation (\ref{eq:120}) for $r = 0,...,n - 1$ we obtain the following Vandermonde's type system of equations
\begin{equation}\label{eq:121}\left(\begin{array}{cccc}
1 & 1 & ... & 1\\
e^{-i\langle v_{1}, \omega_{0}\rangle} & e^{-i\langle v_{2}, \omega_{0}\rangle} & ... & e^{-i\langle v_{n}, \omega_{0}\rangle}\\
& ...............\\
e^{-i(n - 1)\langle v_{1}, \omega_{0}\rangle} & e^{-i(n - 1)\langle v_{2}, \omega_{0}\rangle} & ... & e^{-i(n - 1)\langle v_{n}, \omega_{0}\rangle}\end{array}\right)
\left(\begin{array}{c}D_{v_{1}}(\omega_{0})\\
D_{v_{2}}(\omega_{0})\\
...\\
D_{v_{n}}(\omega_{0})
\end{array}\right)
 = \left(\begin{array}{c}0\\0\\...\\0\end{array}\right).\end{equation}
Observe that non of the terms $D_{v_{i}}(\omega_{0}), 1\leq i\leq n$ can vanish. Indeed, otherwise we would have that $\det(u_{1} - v_{i},...,u_{d} - v_{i}) = 0$ where $v$ is adjacent to the vertices $u_{1},...,u_{d}$ which would imply that $v$ is contained in the hyperplane spanned by $u_{1},...,u_{d}$ and hence it cannot be a vertex of $P$. Hence the determinant of the matrix in the left hand side of equation (\ref{eq:121}) must vanish, i.e., there exist integers $k, k'$ which satisfy $1\leq k, k'\leq n, k\neq k'$ such that $\exp(-i\langle \omega_{0}, v_{k}\rangle) = \exp(-i\langle \omega_{0}, v_{k'}\rangle)$. This is equivalent to the condition that $\langle \omega_{0}, v_{k} - v_{k'}\rangle = 2\pi m$ for some integer $m$. This proves Theorem 2.4.$\square$\\


\textbf{Proof of Theorem 2.5}: During the proof we will use the notation $\mathbf{Proj}_{\Bbb R^{m}}$ to denote the orthogonal projection from $\Bbb R^{d}$ to $\Bbb R^{m}$.

Observe that $S = \{e^{i\pi\langle\cdot, \bar{k}\rangle}:k\in\Bbb Z^{d}\}$ is a set of orthogonal exponentials for the unit cube $C_{d}$. Hence, for every non degenerate matrix $M\in\Bbb R^{d\times d}$ and a point $\bar{k}\in\Bbb Z^{d}\setminus\{\bar{0}\}$ we have
\vskip-0.2cm
\begin{equation}\label{eq:130} 0 = \int_{C_{d}}e^{i\pi\langle x, \bar{k}\rangle}dx = \frac{1}{|\det M|}\int_{MC_{d}}e^{i\pi\langle M^{-1}x, \bar{k}\rangle}dx = \frac{1}{|\det M|}\int_{MC_{d}}e^{i\pi\langle x, (M^{-1})^{T}\bar{k}\rangle}dx.\end{equation}
Let $m$ be an integer satisfying $1\leq m\leq d - 1$ and let us define the following subset $\Gamma = \Gamma(M, m)$ of $\Bbb Z^{d}\setminus\{\bar{0}\}$:
\vskip-0.2cm
$$\hskip-2cm\Gamma(M, m) = \{\bar{k}\in\Bbb Z^{d}\setminus\{\bar{0}\}:\langle (M^{-1})^{T}\bar{k}, e_{i}\rangle = 0, i = m + 1, m + 2,...,d\}.$$
Hence, from equation (\ref{eq:130}) we obtain that for every $\bar{k}$ in $\Gamma(M, m)$ we have
\vskip-0.2cm
$$\hskip-5.5cm 0 = \int_{MC_{d}}e^{i\pi\left[\langle (M^{-1})^{T}\bar{k}, e_{1}\rangle x_{1} + ... + \langle (M^{-1})^{T}\bar{k}, e_{m}\rangle x_{m}\right]}dx$$
$$\hskip-4.25cm = \int_{\mathbf{Proj}_{\Bbb R^{m}}(MC_{d})}e^{i\pi\left[\langle (M^{-1})^{T}\bar{k}, e_{1}\rangle x_{1} + ... + \langle (M^{-1})^{T}\bar{k}, e_{m}\rangle x_{m}\right]}$$
$$\hskip-4.5cm\times\left[\int_{MC_{d}\cap \mathbf{H}_{x_{1},...,x_{m}}}dx_{m + 1}dx_{m + 2}...dx_{d}\right]dx_{1}...dx_{m}.$$
Hence, if we define the weight function
\vskip-0.2cm
$$\hskip-4cm W_{M,m}(x_{1},...,x_{m}) = \int_{MC_{d}\cap \mathbf{H}_{x_{1},...,x_{m}}}dx_{m + 1}dx_{m + 2}...dx_{d}$$
then
\vskip-0.2cm
\begin{equation}\label{eq:131}\hskip-4.25cm\int_{\mathbf{Proj}_{\Bbb R^{m}}(MC_{d})}e^{i\pi\left[\langle (M^{-1})^{T}\bar{k}, e_{1}\rangle x_{1} + ... + \langle (M^{-1})^{T}\bar{k}, e_{m}\rangle x_{m}\right]}\end{equation}
$$\hskip1cm\times W_{M,m}(x_{1},...,x_{m})dx_{1}...dx_{m} = 0, \bar{k}\in\Gamma(M, m).$$
Now, let us define the set $\Lambda$ as follows
$$\hskip-4.4cm\Lambda = \{\bar{0}\}\cup\left\{\pi\left(\langle \bar{k}, M^{-1}e_{1}\rangle,...,\langle \bar{k}, M^{-1}e_{m}\rangle\right):\bar{k}\in\Gamma(M, m)\right\}$$
$$\hskip0.1cm = \left\{\pi\left(\langle \bar{k}, M^{-1}e_{1}\rangle,...,\langle \bar{k}, M^{-1}e_{m}\rangle\right):\bar{k}\in\Bbb Z^{d}, \langle \bar{k}, M^{-1}e_{i}\rangle = 0, i = m + 1, m + 2,...,d\right\}.$$
Then, from equation (\ref{eq:131}) it follows that
\vskip-0.2cm
$$\hskip-4cm\int_{\mathbf{Proj}_{\Bbb R^{m}}(MC_{d})}e^{i\langle \lambda, x^{\ast}\rangle}W_{M, m}(x^{\ast})dx^{\ast} = 0, x^{\ast} = (x_{1},...,x_{m})$$
for every $\lambda\in\Lambda\setminus\{\bar{0}\}$. Also, it can be easily observed that $\Lambda = \Lambda - \Lambda$. Thus, it follows that $E(\Lambda)$ is a set of orthogonal exponentials for the orthogonal projection of $MC_{d}$ on $\Bbb R^{m}$ with respect to the positive weight function $W_{M,m}$.
It is only left to show that the density of $\Lambda$ is greater than or equal to $\pi^{-d}|\det M| / |\det U|$. By our assumption we can find $m$ linearly independent vectors $\bar{k}_{1},...,\bar{k}_{m}$ in $\Bbb Z^{d}\setminus\{\bar{0}\}$ such that
\vskip-0.2cm
\begin{equation}\label{eq:132}\hskip-4.5cm\langle \bar{k}_{j}, M^{-1}e_{i}\rangle = 0, i = m + 1, m + 2,...,d, j = 1,...,m.\end{equation}
Observe also that each linear combination of $\bar{k}_{1},...,\bar{k}_{m}$ in $\Bbb Z$ will produce a vector in $\Bbb Z^{d}$ which satisfy the last system. Hence, we have that $\pi\Lambda'\subset\Lambda$ where
\vskip-0.2cm
$$\hskip-6.4cm\Lambda' = \left\{\alpha_{1}\left(\langle \bar{k}_{1}, M^{-1}e_{1}\rangle,...,\langle \bar{k}_{1}, M^{-1}e_{m}\rangle\right)\right.$$
$$\left. + ... + \alpha_{m}\left(\langle \bar{k}_{m}, M^{-1}e_{1}\rangle,...,\langle \bar{k}_{m}, M^{-1}e_{m}\rangle\right):\alpha_{1},...,\alpha_{m}\in\Bbb Z\right\}.$$
Now, the vectors
\vskip-0.2cm
$$\hskip-5cm\sigma_{i} = \left(\langle \bar{k}_{i}, M^{-1}e_{1}\rangle,...,\langle \bar{k}_{i}, M^{-1}e_{m}\rangle\right), i = 1,...,m$$
are linearly independent. Indeed, otherwise there exist $m$ real numbers $\beta_{1},...,\beta_{m}$, not all zero, such that
\vskip-0.2cm
$$\hskip-9.45cm\bar{0} = \beta_{1}\sigma_{1} + ... + \beta_{m}\sigma_{m}$$
$$\hskip0.5cm = \beta_{1}\left(\langle \bar{k}_{1}, M^{-1}e_{1}\rangle,...,\langle \bar{k}_{1}, M^{-1}e_{m}\rangle\right) + ... + \beta_{m}\left(\langle \bar{k}_{m}, M^{-1}e_{1}\rangle,...,\langle \bar{k}_{m}, M^{-1}e_{m}\rangle\right)$$
$$\hskip-1.8cm = \left(\langle \beta_{1}\bar{k}_{1} + ... + \beta_{m}\bar{k}_{m}, M^{-1}e_{1}\rangle,...,\langle \beta_{1}\bar{k}_{1} + ... + \beta_{m}\bar{k}_{m}, M^{-1}e_{m}\rangle\right).$$
Since the vectors $\bar{k}_{i}, i = 1,...,m$ satisfy the system (\ref{eq:132}) it follows that the vector $\bar{k}_{0} = \beta_{1}\bar{k}_{1} + ... + \beta_{m}\bar{k}_{m}$ will satisfy the system $\langle \bar{k}_{0}, M^{-1}e_{i}\rangle, i = 1,...,d$. But, since $M$ is nondegenerate it follows that $\bar{k}_{0} = 0$ which implies that $\bar{k}_{1},...,\bar{k}_{m}$ are linearly dependent which is a contradiction. Hence, $\Lambda'$ is a full rank lattice and thus has a positive density which is given by $D(\Lambda') = 1 / |\det A|$ where
\vskip0.2cm
$$\hskip-4cm A = \left(\begin{array}{cccc}
\langle \bar{k}_{1}, M^{-1}e_{1}\rangle & \langle \bar{k}_{1}, M^{-1}e_{2}\rangle & ... & \langle \bar{k}_{1}, M^{-1}e_{m}\rangle\\
\langle \bar{k}_{2}, M^{-1}e_{1}\rangle & \langle \bar{k}_{2}, M^{-1}e_{2}\rangle & ... & \langle \bar{k}_{2}, M^{-1}e_{m}\rangle\\
& ..........\\
\langle \bar{k}_{m}, M^{-1}e_{1}\rangle & \langle \bar{k}_{m}, M^{-1}e_{2}\rangle & ... & \langle \bar{k}_{m}, M^{-1}e_{m}\rangle\\
\end{array}\right).$$
In order to calculate the determinant of $A$ observe that
$$\hskip-4.5cm\left(\begin{array}{cc}
A & \bar{0}_{m, d - m}\\
\bar{0}_{d - m, m} & I_{d - m, d - m}
\end{array}\right)
 = \left(\begin{array}{c}
K\\
e_{m + 1}M\\
...\\
e_{d}M
\end{array}\right)M^{-1} = UM^{-1}$$
and thus $\det A = \det U / \det M$. Hence, since $\pi\Lambda'\subseteq\Lambda$ the theorem follows. $\square$\\

In general, an explicit computation of the weight function $W_{M,m}$ is nontrivial. However, in case where $M$ is orthogonal and $m = 1$ or $m = d - 1$ then $W_{M,m}$ can be computed as follows. For $m = 1$, $W_{M,m}$ can be rewritten as
\vskip-0.2cm
$$\hskip-3cm W_{M,m}(x_{1}) = |MC_{d}\cap \mathbf{H}_{x_{1}}| = |C_{d}\cap M^{T}\mathbf{H}_{x_{1}}| = \int_{C_{d}\cap M^{T}\mathbf{H}_{x_{1}}}dS_{x}$$
where $dS_{x}$ is the standard infinitesimal volume measure on the hyperplane $\mathbf{H}_{M}^{x_{1}}: = M^{T}\mathbf{H}_{x_{1}}$. For $\mathbf{H}_{M}^{x_{1}}$ we have the following
\vskip-0.2cm
$$\hskip-0.8cm \mathbf{H}_{M}^{x_{1}} = \left\{u\in\Bbb R^{d}:\left\langle u, \frac{M^{T}e_{1}}{|M^{T}e_{1}|}\right\rangle = \frac{x_{1}}{|M^{T}e_{1}|}\right\} = \left\{u\in\Bbb R^{d}:\left\langle u, M^{T}e_{1}\right\rangle = x_{1}\right\}$$
algebraic presentation where in the last passage we used the fact that $M$ is orthogonal and hence $|M^{T}e_{1}| = 1$. Thus, we can write
\vskip-0.2cm
$$\hskip-8.5cm W_{M,m}(x_{1}) = (\mathcal{R}f)\left(M^{T}e_{1}, x_{1}\right)$$
where $f$ is the characteristic function of the unit cube $C_{d}$ and $(\mathcal{R}f)(\omega, t) = \int_{\langle\omega, x\rangle = t}f(x)dS_{x}$ is the Radon transform of $f$ ($(\omega, t)\in\Bbb S^{d - 1}\times\Bbb R$). Hence, using the Fourier-slice theorem
\vskip-0.2cm
$$\hskip-7.7cm\mathcal{F}[f](u) = \mathcal{F}_{t}\left[(\mathcal{R}f)\left(u / |u|, \cdot\right)\right](|u|),$$
where $\mathcal{F}[f]$ is the Fourier transform of $f$ and $\mathcal{F}_{t}$ is the Fourier transform, with respect to the last variable $t$, of $\mathcal{R}f$, and using the well known formula for the Fourier transform of the unit cube we obtain that
\vskip-0.2cm
$$\hskip-4.5cm W(x_{1}) = \int_{-\infty}^{\infty}\left(\prod_{k = 1}^{d}\frac{\sin\left(2\pi\lambda\langle M^{T}e_{1}, e_{k}\rangle\right)}{\pi\lambda\langle M^{T}e_{1}, e_{k}\rangle}\right)e^{2\pi i\lambda x_{1}}d\lambda.$$
In case where $m = d - 1$ then again by using the assumption that $M$ is orthogonal we have
\vskip-0.2cm
$$\hskip-7cm W_{M, m}(x_{1},...,x_{d - 1}) = \int_{C_{d}\cap M^{T}\mathbf{H}_{x_{1},...,x_{d - 1}}}dl$$
where $dl$ is the infinitesimal length measure of the line
\vskip-0.2cm
$$\hskip-0.65cm l = M^{T}\mathbf{H}_{x_{1},...,x_{d - 1}} = \{u + vt:t\in\Bbb R\}, \hskip0.1cm\textit{where}\hskip0.2cm u = M^{T}\left(\begin{array}{c}x_{1}\\...\\x_{d - 1}\\0\end{array}\right), v = M^{T}e_{d}.$$
For the computation of $W_{M,m}$ first observe that we can assume that non of the components of the direction vector $v$ are equal to zero since otherwise let us assume, without loss of generality, that $v_{m}, v_{m + 1},...,v_{d} = 0, 2\leq m\leq d$. Then, our problem is reduced to finding the volume of the intersection of the line $l$ with the unit cube $C_{m - 1}$ where we think of both $l$ and $C_{m - 1}$ as subsets of the plane $\mathbf{H} = \Bbb R^{m - 1}\times\{(u_{m}, u_{m + 1},...,u_{d})\}$ and where now non of the components of the direction vector of $l$ in $\mathbf{H}$ are equal to zero.

Assuming that the intersection of the line $l$ with the interior of $C_{d}$ is nonempty it follows that there exists a real number $s$ such that $u + sv$ is in the interior of $C_{d}$. A number $s$ which satisfies the last condition can be found by solving the inequalities $u_{i} + sv_{i} > - 1,u_{i} + sv_{i} < 1, i = 1,...,d$. Hence, since the line $l$ can be also parameterized by $u + sv + tv, t\in\Bbb R$ we can assume, without loss of generality, that $u$ is in the interior of $C_{d}$. Also, since $M$ is orthogonal it follows that $|v| = 1$.

Since non of the components of the direction vector $v$ of $l$ are equal to zero it follows that $l$ intersects the boundary of $C_{d}$ at exactly two points $P_{1}$ and $P_{2}$ which are obtained as the images $l(t')$ and $l(t'')$ of some two values $t'$ and $t''$ of $t$. Hence, the length of the segment $C_{d}\cap l$ is equal to
$$\hskip-0.5cm |P_{1} - P_{2}| = |l(t') - l(t'')| = |(u + t'v) - (u + t'')v| = |(t' - t'')v| = |t' - t''|.$$
Hence, all is left is to find the corresponding values $t'$ and $t''$ of $t$. First, observe that $t'$ and $t''$ must have opposite signs since $u$ is strictly inside $C_{d}$ and thus we can assume, without loss of generality, that $t'$ is positive and $t''$ is negative. Observe also that if $l$ intersects the boundary of $C_{d}$ then it must intersect at least two of the following hyperplanes $x_{i} = \pm1, 1\leq i\leq d$. The intersections of $l$ with these hyperplanes occur respectively when $t$ receives the values
\vskip-0.2cm
$$\hskip-6cm t = (1 - u_{i}) / v_{i}, t = - (1 + u_{i}) / v_{i}, 1\leq i\leq d.$$
Hence, the value of $t'$ will be the smallest value of the positive terms from these values of $t$. Indeed, for values of $t$ in $(0, t')$ the segment $l(t), 0 \leq t < t'$ does not intersect any of the hyperplanes $x_{i} = \pm1,1\leq i\leq d$ and hence it is inside $C_{d}$. Thus, at $t = t'$ it must be still inside $C_{d}$ but also on some of the above hyperplanes and thus on the boundary of $C_{d}$. In the same way we can prove that $t''$ is the greatest value of the negative terms of these values. Explicitly, $t'$ and $t''$ are given by
$$\hskip0.5cm t' = 1 / \max\left\{\frac{v_{i}}{1 - u_{i}}, - \frac{v_{i}}{1 + u_{i}}, 1\leq i\leq d\right\}, t'' = 1 / \min\left\{\frac{v_{i}}{1 - u_{i}}, - \frac{v_{i}}{1 + u_{i}}, 1\leq i\leq d\right\}.$$

\begin{example}

Let us choose $d = 3, m = 2$ and the following orthogonal matrix
$$\hskip-3cm M = \frac{1}{\sqrt{q^{2} + 2p^{2}}}\left(\begin{array}{ccc}
q / \sqrt{2} & q / \sqrt{2} & \sqrt{2}p\\
- \sqrt{q^{2} + 2p^{2}} /\sqrt{2} & \sqrt{q^{2} + 2p^{2}} /\sqrt{2} & 0\\
- p & - p & q
\end{array}\right)$$
where $p$ and $q$ are co-prime positive integers. Since the first column of $M^{-1} = M^{T}$ is proportional to an integer vector it follows from Lemma 4.2 that there exist two linearly independent integer vectors $\bar{k}_{1}, \bar{k}_{2}$ satisfying $\langle \bar{k}_{i}, M^{-1}e_{3}\rangle = 0, i = 1, 2$. Hence, the set $E(\Lambda)$ of orthogonal exponentials (with respect to a weight function $W$ to be computed) for the orthogonal projection of $MC_{3}$ on $\Bbb R^{2}$, which was constructed in Theorem 2.5, will have a positive density. In this particular example we can calculate the set $\Lambda$ explicitly
\vskip-0.2cm
$$\hskip-3.25cm\Lambda = \left\{\pi\left(\langle \bar{k}, M^{-1}e_{1}\rangle, \langle \bar{k}, M^{-1}e_{2}\rangle\right):\bar{k}\in\Bbb Z^{3}, \langle \bar{k}, M^{-1}e_{3}\rangle = 0\right\}$$
$$\hskip0.45cm = \left\{\pi\left(\frac{(k_{1} + k_{2})q + 2k_{3}p}{\sqrt{2}\sqrt{q^{2} + 2p^{2}}}, - \frac{1}{\sqrt{2}}(k_{1} - k_{2})\right):\bar{k}\in\Bbb Z^{3}, (k_{1} + k_{2})p - k_{3}q = 0\right\}$$
$$\hskip-3.1cm = \left\{\pi\left(\frac{\sqrt{q^{2} + 2p^{2}}m}{\sqrt{2}}, -\frac{1}{\sqrt{2}}(2n - mq)\right):(n, m)\in\Bbb Z^{2}\right\}$$
and obtain that, since $\Lambda$ is a lattice, its density is given by $D(\Lambda) = \pi^{-2} / \sqrt{q^{2} + 2p^{2}}$. The orthogonal projection of $MC_{3}$ on $\Bbb R^{2}$ can be easily verified to be equal to $P'_{p, q} = \sqrt{2}P_{p, q}$ where $P_{p, q}$ is the following hexagon
\vskip-0.4cm
$$P_{p, q} = \mathrm{Conv}\left\{\pm\left(\frac{p + q}{\sqrt{q^{2} + 2p^{2}}}, 0\right), \pm\left(\frac{p}{\sqrt{q^{2} + 2p^{2}}}, 1\right), \pm\left(\frac{p}{\sqrt{q^{2} + 2p^{2}}}, - 1\right)\right\}.$$
For the computation of the weight function $W$ we need to compute the length of the intersection of the line
\begin{equation}\label{eq:140} M^{T}\mathbf{H}_{x_{1}, x_{2}} = \underset{u}{\underbrace{\frac{1}{\sqrt{2}\sqrt{q^{2} + 2p^{2}}}\left(\begin{array}{c}qx_{1} - \sqrt{q^{2} + 2p^{2}}x_{2}\\qx_{1} + \sqrt{q^{2} + 2p^{2}}x_{2}\\2px_{1}\end{array}\right)}} + \underset{v}{\underbrace{\frac{1}{\sqrt{q^{2} + 2p^{2}}}\left(\begin{array}{c}-p\\-p\\q\end{array}\right)}}t, t\in\Bbb R.\end{equation}
with $C_{3}$. We need to find a real number $s = s(x_{1}, x_{2})$ such that $|u_{i} + sv_{i}| < 1, i = 1,2,3$. Then, as was explained above, $W$ is given by
\vskip-0.2cm
$$\hskip-8cm W(x_{1}, x_{2}) = |t'(x_{1}, x_{2}) - t''(x_{1}, x_{2})|$$
where
\vskip-0.2cm
$$t'(x_{1}, x_{2}) = \frac{1}{\max\left\{\frac{v_{1}}{1 - u_{1} - sv_{1}}, \frac{v_{2}}{1 - u_{2} - sv_{2}}, \frac{v_{3}}{1 - u_{3} - sv_{3}}, - \frac{v_{1}}{1 + u_{1} + sv_{1}}, - \frac{v_{2}}{1 + u_{2} + sv_{2}}, - \frac{v_{3}}{1 + u_{3} + sv_{3}}\right\}},$$
$$t''(x_{1}, x_{2}) = \frac{1}{\min\left\{\frac{v_{1}}{1 - u_{1} - sv_{1}}, \frac{v_{2}}{1 - u_{2} - sv_{2}}, \frac{v_{3}}{1 - u_{3} - sv_{3}}, - \frac{v_{1}}{1 + u_{1} + sv_{1}}, - \frac{v_{2}}{1 + u_{2} + sv_{2}}, - \frac{v_{3}}{1 + u_{3} + sv_{3}}\right\}},$$
and where
$$\hskip-1.15cm u_{1} = \frac{qx_{1} - \sqrt{q^{2} + 2p^{2}}x_{2}}{\sqrt{2}\sqrt{q^{2} + 2p^{2}}},
u_{2} = \frac{qx_{1} + \sqrt{q^{2} + 2p^{2}}x_{2}}{\sqrt{2}\sqrt{q^{2} + 2p^{2}}},
u_{3} = \frac{2px_{1}}{\sqrt{2}\sqrt{q^{2} + 2p^{2}}},$$
$$\hskip-6.7cm v_{1} = v_{2} = - \frac{p}{\sqrt{q^{2} + 2p^{2}}}, v_{3} = \frac{q}{\sqrt{q^{2} + 2p^{2}}}.$$
The computation of $s(x_{1}, x_{2})$ is straightforward but rather technical. Without going into details one can choose $s(x_{1}, x_{2})$ as follows: if we denote for abbreviation $\Delta = \sqrt{q^{2} + 2p^{2}}$ then

\[\hskip-1cm s(x_{1}, x_{2}) =
\begin{cases}
\frac{qx_{1}}{\sqrt{2}p}, \hskip0.1cm\textit{if}\hskip0.1cm\pm\frac{(2p^{2} - q^{2})x_{1}}{\sqrt{2}\triangle} + \frac{q|x_{2}|}{\sqrt{2}} \geq q - p,\\
\frac{\sqrt{2}px_{1}}{q}, \hskip0.1cm\textit{if}\hskip0.2cm \pm\frac{(2p^{2} - q^{2})x_{1}}{\sqrt{2}\triangle} - \frac{q|x_{2}|}{\sqrt{2}} \geq p - q,\\
\frac{\triangle}{2p} - \frac{\triangle}{2q} - \frac{\triangle |x_{2}|}{2\sqrt{2}p} + \frac{\triangle^{2}x_{1}}{2\sqrt{2}pq},  \hskip0.1cm\textit{if}\hskip0.2cm\frac{(2p^{2} - q^{2})x_{1}}{\sqrt{2}\triangle} \pm\frac{q|x_{2}|}{\sqrt{2}} \geq \pm(q - p),\\
-\frac{\triangle}{2p} + \frac{\triangle}{2q} + \frac{\triangle |x_{2}|}{2\sqrt{2}p} + \frac{\triangle^{2}x_{1}}{2\sqrt{2}pq},  \hskip0.1cm\textit{if}\hskip0.2cm \frac{(q^{2} - 2p^{2})x_{1}}{\sqrt{2}\triangle} \pm \frac{q|x_{2}|}{\sqrt{2}} \geq \pm(q - p).\\
\end{cases}
\]\\

\end{example}

\begin{example}

For a more general example in higher dimensions let $d, m, 1\leq m\leq d - 1$, be two positive integers, let $x_{1},...,x_{d}$ be $d$ distinct real numbers such that $x_{m + 1},x_{m + 2},...,x_{d}$ are rational and let $M$ be the $d\times d$ matrix such that $(M^{-1})^{T}$ is the following Vandermonde's type matrix
\vskip-0.2cm
$$\hskip-8cm(M^{-1})^{T} = \left(\begin{array}{cccc}
1 & x_{1} & ... & x_{1}^{d - 1}\\
& ......\\
1 & x_{d} & ... & x_{d}^{d - 1}\\
\end{array}\right).$$
In order to use Theorem 2.5 for the matrix $M$ we need to find $m$ linearly independent integer vectors $\bar{k}_{i}, i = 1,...,m$ such that $\langle (M^{-1})^{T}\bar{k}_{i}, e_{j}\rangle = 0, 1\leq i\leq m, m + 1\leq j\leq d$, or more explicitly the following matrix equation
\begin{equation}\label{eq:141}\hskip-5.85cm\left(\begin{array}{cccc}
1 & x_{1} & ... & x_{1}^{d - 1}\\
& ......\\
1 & x_{m} & ... & x_{m}^{d - 1}\\
1 & x_{m + 1} & ... & x_{m + 1}^{d - 1}\\
& ......\\
1 & x_{d} & ... & x_{d}^{d - 1}\\
\end{array}\right)
\underset{\bar{k}_{i}}{\underbrace{\left(\begin{array}{c}
k_{i,1}\\
...\\
k_{i,m}\\
k_{i,m + 1}\\
...\\
k_{i,d}
\end{array}\right)}}
 = \left(\begin{array}{c}
*\\
...\\
*\\
0\\
...\\
0
\end{array}\right)\end{equation}
should be satisfied for $i = 1,...,m$. Let
\vskip-0.2cm
$$\hskip-1cm p(x) = (x - x_{m + 1})(x - x_{m + 2})...(x - x_{d}) = x^{d - m} + c_{d - m - 1}x^{d - m - 1} + ... + c_{0}$$
be the unique monic polynomial with the roots $x_{m + 1}, x_{m + 2},...,x_{d}$. Observe that since $x_{m + 1},x_{m + 2},...,x_{d}$ are rational then all the coefficients $c_{0},...,c_{d - m - 1}$ of $p$ are also rational. From the definition of the polynomial $p$ it is clear that if we choose
\vskip-0.2cm
$$\hskip-6cm\bar{k}_{i} = \left(\underset{i - 1\hskip0.1cm\textit{times}}{\underbrace{0,...,0}},c_{0},...,c_{d - m - 1},1,\underset{m - i\hskip0.1cm\textit{times}}{\underbrace{0,...,0}}\right)^{T}$$
then the system (\ref{eq:141}) will be satisfied. It can be easily seen that with this choice the vectors $\bar{k}_{i}, 1\leq i\leq m$ are linearly independent and while they will be rational we can just multiply each of them by the common denominator of all their rational components in order to convert them to integer vectors. Finally, we can use $\bar{k}_{i}, 1\leq i\leq m$ to build an infinite set of orthogonal exponentials with a positive density since, as was shown in the proof of Theorem 2.5, the set
$$\hskip-6.4cm\pi\Lambda' = \left\{\alpha_{1}\pi\left(\langle \bar{k}_{1}, M^{-1}e_{1}\rangle,...,\langle \bar{k}_{1}, M^{-1}e_{m}\rangle\right)\right.$$
$$\left. + ... + \alpha_{m}\pi\left(\langle \bar{k}_{m}, M^{-1}e_{1}\rangle,...,\langle \bar{k}_{m}, M^{-1}e_{m}\rangle\right):\alpha_{1},...,\alpha_{m}\in\Bbb Z\right\}$$
will have a positive density and $E(\pi\Lambda')$ will be a set of orthogonal exponentials for the orthogonal projection of $MC_{d}$ on $\Bbb R^{m}$.\\

\hskip-0.625cm For example, let us choose $d = 5, m = 3$, $x_{1} = 0, x_{2} = 1, x_{3} = - 1, x_{4} = 2, x_{5} = - 2$ and the matrix $M$ such that
\vskip-0.2cm
$$\hskip-8cm(M^{-1})^{T} = \left(\begin{array}{ccccc}
1 & 0 & 0 & 0 & 0\\
1 & 1 & 1 & 1 & 1\\
1 & - 1 & 1 & - 1 & 1\\
1 & 2 & 4 & 8 & 16\\
1 & -2 & 4 & -8 & 16\\
\end{array}\right).$$
The polynomial $p$ in this case is given by $p(x) = (x - 2)(x + 2) = x^{2} - 4$ which produces the following vectors
\vskip-0.2cm
$$\hskip-6cm\bar{k}_{1} = \left(\begin{array}{c}- 4\\ 0\\ 1\\ 0\\0\end{array}\right),
\bar{k}_{2} = \left(\begin{array}{c}0 \\ - 4\\ 0\\ 1\\ 0\end{array}\right),
\bar{k}_{3} = \left(\begin{array}{c}0 \\ 0 \\ - 4\\ 0\\ 1\end{array}\right).$$
Hence, $E(\pi\Lambda')$ will be a set of orthogonal exponentials for the orthogonal projection of $MC_{5}$ on $\Bbb R^{3}$ where
$$\hskip-2.25cm\Lambda' = \{\alpha_{1}(4, 3, 3) + \alpha_{2}(0, 3, - 3) + \alpha_{3}(0, 3, 3):\alpha_{1},\alpha_{2},\alpha_{3}\in\Bbb Z\}.$$
Now, we are left with computing the orthogonal projection of $MC_{5}$ on $\Bbb R^{3}$. By a direct compuation we have that
\vskip-0.2cm
$$\hskip-5.5cm M = \left(\begin{array}{ccccc}
1 & 0 & -5/4 & 0 & 1 / 4\\
0 & 2/3 & 2/3 & -1/6 & - 1 / 6\\
0 & -2/3 & 2/3 & 1/6 & -1 / 6\\
0 & -1/12 & -1/24 & 1/12 & 1 / 24\\
0 & 1/12 & -1/24 & -1/12 & 1 / 24
\end{array}\right).$$


\begin{figure}[t]

    \caption{}

    \centering

    \begin{tikzpicture}

        \begin{axis}[xtick={0}, xlabel = {$X$}, ytick={0}, ylabel = {$Y$}, ztick={0}, view = {10}{60}, scale = 1.5]

            \addplot3 [samples = 5,domain = 0:1]
            (
                {x * (-2.5) + (1 - x) * (-0.5)}, {x * (1.664) + (1 - x) * (1.664)}, {x * (0) + (1 - x) * (0)}
            );
            \addplot3 [samples = 5,domain = 0:1]
            (
                {x * (-2.5) + (1 - x) * (-0.5)}, {x * (0) + (1 - x) * (0)}, {x * (1.664) + (1 - x) * (1.664)}
            );
            \addplot3 [samples = 5,domain = 0:1]
            (
                {x * (2.5) + (1 - x) * (0.5)}, {x * (-1.664) + (1 - x) * (-1.664)}, {x * (0) + (1 - x) * (0)}
            );
            \addplot3 [samples = 5,domain = 0:1]
            (
                {x * (2.5) + (1 - x) * (0.5)}, {x * (0) + (1 - x) * (0)}, {x * (-1.664) + (1 - x) * (-1.664)}
            );
            \addplot3 [samples = 5,domain = 0:1]
            (
                {x * (-2.5) + (1 - x) * (-2.5)}, {x * (0) + (1 - x) * (1.664)}, {x * (1.664) + (1 - x) * (0)}
            );
            \addplot3 [samples = 5,domain = 0:1]
            (
                {x * (-0.5) + (1 - x) * (-0.5)}, {x * (1.664) + (1 - x) * (0)}, {x * (0) + (1 - x) * (1.664)}
            );
            \addplot3 [samples = 5,domain = 0:1]
            (
                {x * (2.5) + (1 - x) * (2.5)}, {x * (0) + (1 - x) * (-1.664)}, {x * (-1.664) + (1 - x) * (0)}
            );
            \addplot3 [samples = 5,domain = 0:1]
            (
                {x * (0.5) + (1 - x) * (0.5)}, {x * (-1.664) + (1 - x) * (0)}, {x * (0) + (1 - x) * (-1.664)}
            );
            \addplot3 [samples = 5,domain = 0:1]
            (
                {x * (-2.5) + (1 - x) * (-2)}, {x * (1.664) + (1 - x) * (1.332)}, {x * (0) + (1 - x) * (-0.332)}
            );
             \addplot3 [samples = 5,domain = 0:1]
            (
                {x * (-2.5) + (1 - x) * (-2)}, {x * (0) + (1 - x) * (-0.332)}, {x * (1.664) + (1 - x) * (1.332)}
            );
            \addplot3 [samples = 5,domain = 0:1]
            (
                {x * (2.5) + (1 - x) * (2)}, {x * (-1.664) + (1 - x) * (-1.332)}, {x * (0) + (1 - x) * (0.332)}
            );
            \addplot3 [samples = 5,domain = 0:1]
            (
                {x * (2.5) + (1 - x) * (2)}, {x * (0) + (1 - x) * (0.332)}, {x * (-1.664) + (1 - x) * (-1.332)}
            );
            \addplot3 [samples = 5,domain = 0:1]
            (
                {x * (-2) + (1 - x) * (-2)}, {x * (-0.332) + (1 - x) * (1.332)}, {x * (1.332) + (1 - x) * (-0.332)}
            );
            \addplot3 [samples = 5,domain = 0:1]
            (
                {x * (2) + (1 - x) * (2)}, {x * (0.332) + (1 - x) * (-1.332)}, {x * (-1.332) + (1 - x) * (0.332)}
            );
             \addplot3 [samples = 5,domain = 0:1]
            (
                {x * (0.5) + (1 - x) * (-2)}, {x * (-1.664) + (1 - x) * (-0.332)}, {x * (0) + (1 - x) * (1.332)}
            );
            \addplot3 [samples = 5,domain = 0:1]
            (
                {x * (-0.5) + (1 - x) * (2)}, {x * (1.664) + (1 - x) * (0.332)}, {x * (0) + (1 - x) * (-1.332)}
            );
            \addplot3 [samples = 5,domain = 0:1]
            (
                {x * (0.5) + (1 - x) * (-2)}, {x * (0) + (1 - x) * (1.332)}, {x * (-1.664) + (1 - x) * (-0.332)}
            );
            \addplot3 [samples = 5,domain = 0:1]
            (
                {x * (-0.5) + (1 - x) * (2)}, {x * (0) + (1 - x) * (-1.332)}, {x * (1.664) + (1 - x) * (0.332)}
            );
        \end{axis}

   \end{tikzpicture}

\end{figure}


Hence, using the factorization $C_{5} = [-e_{1}, e_{1}] + ... + [-e_{5}, e_{5}]$ of the unit cube we have that
\vskip-0.2cm
$$\hskip-3.25cm P = \mathbf{Proj}_{\Bbb R^{3}}(MC_{5}) = \mathbf{Proj}_{\Bbb R^{3}}(M([-e_{1}, e_{1}] + ... + [-e_{5}, e_{5}]))$$
$$\hskip-5.8cm = \mathbf{Proj}_{\Bbb R^{3}}([-Me_{1}, Me_{1}] + ... + [-Me_{5}, Me_{5}])$$
$$\hskip-4.2cm = \mathbf{Proj}_{\Bbb R^{3}}([-Me_{1}, Me_{1}]) + ... + \mathbf{Proj}_{\Bbb R^{3}}([-Me_{5}, Me_{5}])$$
$$\hskip-0.95cm = [-\mathbf{Proj}_{\Bbb R^{3}}(Me_{1}), \mathbf{Proj}_{\Bbb R^{3}}(Me_{1})] + ... + [-\mathbf{Proj}_{\Bbb R^{3}}(Me_{5}), \mathbf{Proj}_{\Bbb R^{3}}(Me_{5})]$$
$$\hskip-3.75cm = \left[-(1, 0, 0), (1, 0, 0)\right] + \left[-\left(0, 2/3, -2/3\right), \left(0, 2/3, -2/3\right)\right]$$
$$\hskip-0.5cm + \left[-\left(-5/4, 2/3, 2/3\right), \left(-5/4, 2/3, 2/3\right)\right] + \left[- \left(0, -1/6, 1/6\right), \left(0, -1/6, 1/6\right)\right]$$
$$\hskip-6.25cm + \left[-\left(1/4, -1/6, -1/6\right), \left(1/4, -1/6, -1/6\right)\right]$$
$$\hskip-3.45cm = \mathrm{Conv}\left\{
\pm\left(\frac{1}{2}, -\frac{5}{3}, 0\right),
\pm\left(-\frac{5}{2}, 0, \frac{5}{3}\right),
\pm\left(-2, -\frac{1}{3}, \frac{4}{3}\right),\right.$$
$$\hskip-1.6cm\left.\pm\left(\frac{1}{2}, 0, -\frac{5}{3}\right),
\pm\left(-\frac{5}{2}, \frac{5}{3}, 0\right),
\pm\left(-2, \frac{4}{3}, -\frac{1}{3}\right)\right\}$$
(see Figure 2 for the polytope $P$).

\end{example}

\section{Appendix}

\begin{prop}

Let $P$ be a simple, convex and rational polytope. Let $E(\Lambda)$ be any infinite set of orthogonal exponentials produced by the construction used in the proof of Theorem 2.1. Then, if the density of $\Lambda$ exists then it must be zero.

\end{prop}

\begin{proof}

In order to make the proof more transparent we will assume for simplicity that $d = 2$. The proof in higher dimensions is similar. Also, observe that from the construction of the set $\Lambda$, given in the proof of Theorem 2.1, it all ways follows that $\Lambda\subseteq 2\pi\Bbb Z^{2}$. The reader should take note that we will implicitly be using this fact during the proof of Proposition 4.1.

Observe that we can assume that all the vertices of $P$ belong to $\Bbb Z^{2}$. Indeed, if $N$ is a common denominator for all the rational components in all the vertices of $P$ then the affine transformation $x\mapsto Nx$ will map $P$ to a polytope $P'$ whose vertices are all in $\Bbb Z^{2}$. Hence, if $\Lambda$ is the set constructed for the polytope $P$, by using the method in the proof of Theorem 2.1, and $\Lambda'$ is the corresponding set constructed for $P'$, using the exact same construction, then $D(\Lambda') = N^{2}D(\Lambda)$. Hence, the density of $\Lambda$ is zero if and only if the same is true for $\Lambda'$.

Thus, we can assume that the lattice $L$, given as in the proof of Theorem 2.1, is equal to $2\pi\Bbb Z^{2}$.

Now, assume by contradiction that $\Lambda$ has a positive density. Then, by the definition of the density $D(\Lambda)$ we have in particular that
\vskip-0.2cm
$$\hskip-5.75cm D(\Lambda) = \lim_{\rho\rightarrow\infty}\inf_{x\in\Bbb R^{2}}\frac{|\Lambda\cap(x + S_{\rho})|}{\rho^{2}} = C > 0$$
where $S_{\rho} = [0, \rho)\times[0, \rho)$. Hence, there exists a positive real number $\rho_{0}$ such that for every $\rho\geq\rho_{0}$ and $x\in\Bbb R^{2}$ we have $|\Lambda\cap(x + S_{\rho})| > 0$. Let $N$ be an integer which is greater or equal to $\rho_{0}$. Then, it follows that for every $x$ in $\Bbb R^{2}$ the set $\Lambda\cap(x + S_{2\pi N})$ is nonempty. In particular, for every $(n, m)\in\Bbb Z^{2}$ the set $B_{n,m} = \Lambda\cap((2\pi nN, 2\pi mN) + S_{2\pi N})$ is nonempty. Hence, there exists a point $p_{n,m}\in B_{n,m}$ and from the definition of $B_{n,m}$ the point $p_{n,m}$ has the form $p_{n,m} = (2\pi nN, 2\pi mN) + v_{n, m}$ where $v_{n, m}\in S_{2\pi N}\cap 2\pi\Bbb Z^{2}$. Since $B_{n,m}\subset \Lambda$ it follows that $p_{n,m}\in \Lambda$ for every $(n,m)\in\Bbb Z^{2}$.

Now, let us recall that in the construction of $\Lambda$, which was used in the proof of Theorem 2.1, for each point $p$ we chose in $2\pi\Bbb Z^{2}$ we had to omit (except for $p$) all the points in $2\pi\Bbb Z^{2}$ which are on the lines which are orthogonal to any of the edges of the polytope $P$ and which pass through $p$. Observe that since all of the vertices of $P$ have integer components then it follows that each line $l$, which is orthogonal to one of the edges of $P$ and which passes through $p$, has a parametrization of the form
\vskip-0.2cm
$$\hskip-8.7cm l = p + (2\pi u, 2\pi v)t, t\in\Bbb R$$
where $u$ and $v$ are integers. Let us take some integers $u$ and $v$ such that the vector $(u, v)$ is orthogonal to one of the edges of $P$ and observe that since $(u, v)\neq (0, 0)$ we can assume, without loss of generality, that $v\neq0$.

Hence, it follows that for each point $p_{n,m}$, as was constructed above, we have to omit the line $l_{n,m}:= p_{n,m} + (2\pi u, 2\pi v)t, t\in\Bbb R$ from $2\pi\Bbb Z^{2}$ (except for the point $p_{n,m}$). Our aim now is to use this fact in order to obtain a contradiction, i.e, we will show that it cannot be that all the points $p_{n,m}, (n, m)\in\Bbb Z^{2}$ are in $\Lambda$ if for each point $p_{n,m}$ we omit the line $l_{n,m}$ (except for the point $p_{n,m}$) from $2\pi\Bbb Z^{2}$.

For this, let us fix an integer $m$ and let us consider the set $\Omega_{m}\subseteq 2\pi\Bbb Z^{2}$ which is omitted after choosing the points $p_{n, m}, n\in\Bbb Z$ and then omitting the corresponding lines $l_{n,m}, n\in\Bbb Z$ from $2\pi\Bbb Z^{2}$ (except for the points $p_{n,m}$). We claim that $\Omega_{m}$ has a positive density. Indeed, by the construction of the points $p_{n,m}$ it follows that $\Omega_{m}$ is given by
\vskip-0.2cm
$$\hskip-5.75cm\Omega_{m} = \{p_{n,m} + k(2\pi u, 2\pi v):n\in\Bbb Z, k\in\Bbb Z\setminus\{0\}\}$$
\begin{equation}\label{eq:200}\hskip-1.85cm = \{(2\pi nN, 2\pi mN) + k(2\pi u, 2\pi v) + v_{n, m}:n\in\Bbb Z, k\in\Bbb Z\setminus\{0\}\}.\end{equation}
First, observe that in the set which defines $\Omega_{m}$, in the right hand side of equation (\ref{eq:200}), there are no repeated points. That is, if $(n, k), (n', k')\in\Bbb Z\times(\Bbb Z\setminus\{0\})$ and $(n, k)\neq (n', k')$ then
\begin{equation}\label{eq:201}(2\pi nN, 2\pi mN) + k(2\pi u, 2\pi v) + v_{n, m}\neq (2\pi n'N, 2\pi mN) + k'(2\pi u, 2\pi v) + v_{n', m}.\end{equation}
Indeed, first observe that for every two different points $(a, b), (a', b')\in\Bbb Z^{2}$, $(a, b)\neq(a', b')$ we have that $p_{a, b}\neq p_{a', b'}$ since otherwise we will have
$$\hskip-5cm (2\pi aN, 2\pi bN) +  v_{a, b} = (2\pi a'N, 2\pi b'N) +  v_{a', b'}.$$
But, since $v_{a, b}, v_{a', b'}\in 2\pi\Bbb Z^{2}\cap S_{2\pi N}$, we will have that
$$(2\pi (a - a')N, 2\pi (b - b')N) =  v_{a', b'} - v_{a, b}\in S_{2\pi N} - S_{2\pi N}\subseteq (- 2\pi N, 2\pi N)\times(- 2\pi N, 2\pi N)$$
which will imply that $a = a', b = b'$ in contrast to our assumption. Hence, if there is equality in equation (\ref{eq:201}) then if $k = k'$ then we will have that $p_{n, m} = p_{n', m}$ which, as was explained above, will imply that $n = n'$ and thus $(n, k) = (n', k')$ in contrast to our assumption. If $k\neq k'$ then we will have the equality $p_{n', m} = p_{n, m} + (k - k')(2\pi u, 2\pi v)$ which will imply that $p_{n', m}$ is a point that must be omitted after choosing the point $p_{n, m}$ which is a contradiction to the assumption that $p_{n', m}\in \Lambda$.

Thus, since there are no repeated points in the set in the right hand side of equation (\ref{eq:200}) which defines $\Omega_{m}$ and since the points $v_{n,m}$ are bounded, i.e., belong to $S_{2\pi N}\cap 2\pi\Bbb Z^{2}$, it easily follows that $\Omega_{m}$ has the same density as the set
$$\hskip-4.3cm\Omega' = \{(2\pi nN, 2\pi mN) + k(2\pi u, 2\pi v):n\in\Bbb Z, k\in\Bbb Z\}$$
$$\hskip-1.85cm = \{(0, 2\pi mN)\} + \underset{\Omega''}{\underbrace{\{(2\pi N, 0)n + (2\pi u, 2\pi v)k:n\in\Bbb Z, k\in\Bbb Z\}}}.$$
Since $v\neq0$ then $\Omega''$ is a full rank lattice whose density is positive and is given explicitly by $1 / (4\pi^{2}vN)$ and obviously $\Omega'$ and $\Omega''$ have the same density. Hence, $\Omega_{m}$ has a positive density which is equal to $1 / (4\pi^{2}vN)$  and does not depend on $m$.

Next, observe that all the sets $\Omega_{m}, m\in\Bbb Z$ are disjoint. Indeed, otherwise there exist two different integers $m$ and $m'$ such that
$$\hskip-6cm p_{n,m} + k(2\pi u, 2\pi v) = p_{n' ,m'} + k'(2\pi u, 2\pi v)$$
for some $n, n'\in\Bbb Z, k, k'\in\Bbb Z\setminus\{0\}$. If $k = k'$ then $p_{n, m} = p_{n', m'}$ which will imply that $m = m'$ in contrast to our assumption. If $k\neq k'$ then $p_{n', m'} = p_{n, m} + (k - k')(2\pi u, 2\pi v)$ which will imply that $p_{n',m'}$ is a point that must be omitted after choosing the point $p_{n, m}$ which is a contradiction to the assumption that it belongs to $\Lambda$.

Now, we easily obtain a contradiction since all the sets $\Omega_{m}\subseteq 2\pi\Bbb Z^{2}, m\in\Bbb Z$ are disjoint and have the same positive density but this will imply that $D(2\pi\Bbb Z^{2}) = \infty$ which is not true.

\end{proof}

Before proving Lemma 4.2 we make the following notation. If $M$ is an $n\times m$ matrix whose rows are given by the vectors $\bar{v}_{1},..,\bar{v}_{n}$ then we denote by $\bar{v}^{(1)},...,\bar{v}^{(m)}$ the column vectors of $M$, i.e., $\bar{v}^{(k)} = Me_{k}, 1\leq k\leq m$.

\begin{lem}

Let $\bar{v}_{i}, i = 1,...,d'$, $1\leq d'\leq d - 1,$ be linearly independent vectors in $\Bbb R^{d}$. Then, the following two conditions are equivalent:\\

$\bullet$ There exist $d - d'$ linearly independent (with respect to $\Bbb R$) integer vectors $\bar{k}_{i}, i = 1,...,d - d',$ which are orthogonal to each one of the vectors $\bar{v}_{1},...,\bar{v}_{d'}$.\\

$\bullet$ There exists a subset $S = \{i_{1},...,i_{d - d'}\}\subset\{1,...,d\}$, with the complement $\{1,...,d\}\setminus S = \{j_{1},...,j_{d'}\}$, such that the column vectors $\bar{v}^{(j_{1})},...,\bar{v}^{(j_{d'})}$ are linearly independent and such that the $d'\times(d - d')$ matrix $(\bar{v}^{(j_{1})},...,\bar{v}^{(j_{d'})})^{-1}(\bar{v}^{(i_{1})},...,\bar{v}^{(i_{d - d'})})$ is rational.

\end{lem}

\begin{proof}

Let us prove the first direction. That is, we assume that there exist $d - d'$ linearly independent (with respect to $\Bbb R$) integer vectors $\bar{k}_{i}, i = 1,...,d - d',$ which are orthogonal to each one of the vectors $\bar{v}_{1},...,\bar{v}_{d'}$ and our aim is to show the existence of the set $S$ as described in the formulation of Lemma 4.2.

Observe that since $\mathrm{rank}\{\bar{k}_{1},...,\bar{k}_{d - d'}\} = d - d'$ then there must be $d - d'$ column vectors $\bar{k}^{(i_{m})}, m = 1,...,d - d'$ which are linearly independent. Then, we choose $S = \{i_{1},...,i_{d - d'}\}$. Now, we claim that the column vectors $\bar{v}^{(j_{m})}, m = 1,...,d'$ are linearly independent. Indeed, the condition that $\bar{k}_{1},...,\bar{k}_{d - d'}$ are orthogonal to $\bar{v}_{1},...,\bar{v}_{d'}$ can explicitly be written by
$$\hskip-3.2cm k_{1, i_{1}}\bar{v}^{(i_{1})} + ... + k_{1,i_{d - d'}}\bar{v}^{i_{(d - d')}} = - k_{1, j_{1}}\bar{v}^{(j_{1})} - ... - k_{1, j_{d'}}\bar{v}^{(j_{d'})},$$
\begin{equation}\label{eq:210}\hskip-9cm.......................\end{equation}
$$\hskip-1.35cm k_{d - d', i_{1}}\bar{v}^{(i_{1})} + ... + k_{d - d',i_{d - d'}}\bar{v}^{i_{(d - d')}} = - k_{d - d', j_{1}}\bar{v}^{(j_{1})} - ... - k_{d - d', j_{d'}}\bar{v}^{(j_{d'})}.$$
Since $\bar{k}^{(i_{1})},...,\bar{k}^{(i_{d - d'})}$ are linearly independent it follows that the coefficients $(d - d')\times(d - d')$ matrix $\{k_{i,j}\}_{i = 1,...,d - d', j = i_{1},...,i_{d - d'}}$, in the left hand side of the last system of equations, is nondegenerate. Hence, $\bar{v}^{(i_{1})},...,\bar{v}^{(i_{d - d'})}$ can be expressed in terms of $\bar{v}^{(j_{1})},...,\bar{v}^{(j_{d'})}$ and thus these vectors must be linearly independent (since $\mathrm{rank}\{\bar{v}^{(1)},...,\bar{v}^{(d)}\} = d'$). The system of equations (\ref{eq:210}) has the following matrix form
$$\hskip-4.5cm\underset{A}{\underbrace{\left(\begin{array}{ccc}
v_{1, j_{1}} & ... & v_{1, j_{d'}}\\
& ........ \\
v_{d', j_{1}} & ... & v_{d', j_{d'}}
\end{array}\right)}}
\left(\begin{array}{ccc}
k_{1, j_{1}} & ... & k_{d - d', j_{1}}\\
& ........ \\
k_{1, j_{d'}} & ... & k_{d - d', j_{d'}}
\end{array}\right)$$
$$\hskip-2cm = - \left(\begin{array}{ccc}
v_{1, i_{1}} & ... & v_{1, i_{d - d'}}\\
& ........ \\
v_{d', i_{1}} & ... & v_{d', i_{d - d'}}
\end{array}\right)
\underset{B}{\underbrace{\left(\begin{array}{ccc}
k_{1, i_{1}} & ... & k_{d - d', i_{1}}\\
& ........ \\
k_{1, i_{d - d'}} & ... & k_{d - d', i_{d - d'}}
\end{array}\right)}},$$
where the matrices $A$ and $B$ are invertible. Hence, we obtained that
$$\hskip-4cm\left(\begin{array}{ccc}
v_{1, j_{1}} & ... & v_{1, j_{d'}}\\
& ........ \\
v_{d', j_{1}} & ... & v_{d', j_{d'}}
\end{array}\right)^{-1}\left(\begin{array}{ccc}
v_{1, i_{1}} & ... & v_{1, i_{d - d'}}\\
& ........ \\
v_{d', i_{1}} & ... & v_{d', i_{d - d'}}
\end{array}\right)$$
$$ = \left(\begin{array}{ccc}
k_{1, j_{1}} & ... & k_{d - d', j_{1}}\\
& ........ \\
k_{1, j_{d'}} & ... & k_{d - d', j_{d'}}
\end{array}\right)\left(\begin{array}{ccc}
k_{1, i_{1}} & ... & k_{d - d', i_{1}}\\
& ........ \\
k_{1, i_{d - d'}} & ... & k_{d - d', i_{d - d'}}
\end{array}\right)^{-1}\in\Bbb Q^{d'\times(d - d')}.$$
On the other hand, suppose that there exists a subset $S = \{i_{1},...,i_{d - d'}\}\subset\{1,...,d\}$ such that $\{1,...,d\}\setminus S = \{j_{1},...,j_{d'}\}$ and such that the column vectors $\bar{v}^{(j_{1})},...,\bar{v}^{(j_{d'})}$ are linearly independent and suppose that there exists a $d'\times(d - d')$ rational matrix $K = \{k_{i,j}\}_{i = 1,...,d', j = 1,...,d - d'}$ such that
$$\left(\begin{array}{ccc}
v_{1, j_{1}} & ... & v_{1, j_{d'}}\\
& ........ \\
v_{d', j_{1}} & ... & v_{d', j_{d'}}
\end{array}\right)^{-1}\left(\begin{array}{ccc}
v_{1, i_{1}} & ... & v_{1, i_{d - d'}}\\
& ........ \\
v_{d', i_{1}} & ... & v_{d', i_{d - d'}}
\end{array}\right) = \left(\begin{array}{ccc}
k_{1, 1} & ... & k_{1, d - d'}\\
 & .......\\
k_{d', 1} & ... & k_{d', d - d'}\\
\end{array}\right).$$
Our aim is to show that there exist linearly independent integer vectors $\bar{k}_{1},...,\bar{k}_{d - d'}$ which are orthogonal to $\bar{v}_{1},...,\bar{v}_{d'}$. For this, observe that if we can find rational matrices $K' = \{k_{i,j}'\}_{i = 1,...,d - d', j = 1,..,d'}$ and $K'' = \{k_{i,j}''\}_{i = 1,...,d - d', j = 1,...,d - d'}$ such that $K''$ is invertible and such that
$$\left(\begin{array}{ccc}
k_{1, 1}' & ... & k_{d - d', 1}'\\
& ........ \\
k_{1, d'}' & ... & k_{d - d', d'}'
\end{array}\right)\left(\begin{array}{ccc}
k_{1, 1}'' & ... & k_{d - d', 1}''\\
& ........ \\
k_{1, d - d'}'' & ... & k_{d - d', d - d'}''
\end{array}\right)^{-1} = \left(\begin{array}{ccc}
k_{1, 1} & ... & k_{1, d - d'}\\
 & .......\\
k_{d', 1} & ... & k_{d', d - d'}\\
\end{array}\right),$$
then we can just reverse our argument in the proof of the first direction in order to show the existence of the integer vectors $\bar{k}_{1},...,\bar{k}_{d - d'}$ as described above. Indeed, by reversing our argument we will obtain rational vectors $\bar{k}_{1},...,\bar{k}_{d - d'}$ such that the system (\ref{eq:210}) is satisfied which is equivalent to the fact that these vectors are orthogonal to $\bar{v}_{1},...,\bar{v}_{d'}$. In order to obtain integer vectors we can just multiply each such rational vector $\bar{k}_{i}, 1\leq i\leq d - d'$, by the common denominator of all the rational components of these vectors. The fact that these integer vectors are linearly independent follows easily from the fact that $K''$ in invertible.

As for the matrices $K'$ and $K''$ we can choose $K''$ to be the identity matrix and choose $k_{i,j}' := k_{j,i}$ for $i = 1,...,d - d', j = 1,...,d'$.

\end{proof}

\end{document}